\documentclass{article}

\usepackage{fullpage}
\usepackage{authblk}

\usepackage{graphicx,amsfonts} 
\usepackage{tikz}
\usepackage{algorithm,nicematrix}
\usepackage{algpseudocodex}
\usepackage{xcolor}
\usepackage{multirow,hyperref}

\hypersetup{
    colorlinks=true,
    allcolors=blue,
}
\usepackage[capitalize,noabbrev]{cleveref}

\crefalias{ALG@line}{line}
\makeatletter
\AddToHook{env/algorithmic/begin}{\def\@currentcounter{ALG@line}}
\makeatother

\usepackage{amsmath,amssymb,amsthm}
\usepackage{thmtools}
\usepackage{thm-restate}

\usetikzlibrary{patterns,tikzmark}
\usetikzlibrary{matrix,decorations.pathreplacing,calc,fit,backgrounds,calligraphy}

\algrenewcommand\algorithmiccomment[1]{\hfill\textit{\textcolor{gray}{{ $\triangleright$  #1}}}}

\newcommand{\field}{\mathbb{F}}
\newcommand{\matsize}[2]{#1\times #2}
\newcommand{\matexp}{\omega}

\newcommand{\dgm}{\mathrm{dgm}}
\newcommand{\mat}[1]{\mathbf{#1}}
\newcommand{\bdr}{\mat{D}}
\newcommand{\PH}{PH}

\newcommand{\low}{\mathbf{low}}
\newcommand{\lft}{\mathbf{lft}}

\NewCommandCopy{\standardparagraph}{\paragraph}
\renewcommand{\paragraph}[1]{\standardparagraph*{#1.}}

\DeclareMathOperator*{\argmin}{arg\,min}
\DeclareMathOperator*{\argmax}{arg\,max}

\pgfmathsetmacro{\myscale}{2}
\pgfkeys{tikz/mymatrixenv/.style={decoration={brace},every left delimiter/.style={xshift=8pt},every right delimiter/.style={xshift=-8pt}}}
\pgfkeys{tikz/mymatrix/.style={matrix of math nodes,nodes in empty cells,
left delimiter={[},right delimiter={]},inner sep=1pt,outer sep=1.5pt,
column sep=8pt,row sep=8pt,nodes={minimum width=20pt,minimum height=10pt,
anchor=center,inner sep=0pt,outer sep=0pt,scale=\myscale,transform shape}}}

\tikzset{reddish/.style={
    fill=red,draw opacity=0.4,
    draw=red,fill opacity=0.3,
  }}
\newtheorem{theorem}{Theorem}[section]
\newtheorem{lemma}[theorem]{Lemma}

\newtheorem{remark}[theorem]{Remark}
\newtheorem{observation}[theorem]{Observation}

\newcommand\blfootnote[1]{%
  \begingroup
  \renewcommand\thefootnote{}\footnote{#1}%
  \addtocounter{footnote}{-1}%
  \endgroup
}

\title{Persistent (Co)Homology in Matrix Multiplication Time}
\author[1]{Dmitriy Morozov}
\author[2,3]{Primoz Skraba}

\affil[1]{Lawrence Berkeley National Laboratory}
\affil[2]{School of Mathematical Sciences, Queen Mary University of London}
\affil[3]{Faculty of Computer Science, University of Ljubljana}
\begin{document}

\maketitle
\begin{abstract}
Most algorithms for computing persistent homology do so by tracking cycles that represent homology classes. There are many choices of such cycles, and
specific choices have found different uses in applications.
Although it is known that persistence diagrams can be computed in matrix multiplication time for the more general case of zigzag persistent homology~\cite{milosavljevic2011zigzag}, it is not clear how to extract cycle representatives, especially if specific representatives are desired.
In this paper, we provide the same matrix multiplication bound for computing representatives for the two choices common in applications in the case of ordinary persistent (co)homology.
We first provide a fast version of the reduction algorithm,  which is simpler than the algorithm in \cite{milosavljevic2011zigzag}, but returns a different set of representatives than the standard algorithm~\cite{edelsbrunner2002topological}.
We then give a fast version of a  variant called the row algorithm \cite{de2011dualities}, which returns the same representatives as the standard algorithm. 
\end{abstract}


\section{Introduction}

Persistent homology~\cite{edelsbrunner2002topological} is one of the core techniques in topological data analysis. Because of its theoretical importance as well as its wide use in applications, a lot of work has been dedicated to its efficient computation both in theory \cite{milosavljevic2011zigzag,de2011dualities,chen2011output,busaryev2012annotating,kerber2016persistent} and in practice~\cite{bauer2014clear,bauer2017phat,bauer2021ripser,hylton2017performance,henselmanghristl6,wagner2011efficient}.
On the theoretical side, the connection between persistent homology and Gaussian elimination was established early on~\cite{zomorodian2004computing}, and as a result it has been long accepted that persistence can be computed in matrix multiplication time. Milosavljevic et al.~\cite{milosavljevic2011zigzag} gave an explicit algorithm for the more general case of zigzag persistence~\cite{CdS10,CdSM09}. The main difficulty in applying standard linear algebra techniques is the strict constraint on the row and column ordering of the boundary matrix. 
\blfootnote{
This paper is the extended version of \cite{morozov24}. We discovered that the fast row algorithm in that paper is incorrect. The present paper provides the correct algorithm along with other improvements.
}

As persistent homology has evolved as a field, the significance of its formulation as an $\mat{R} = \mat{D} \mat{V}$ decomposition~\cite{cohen2006vines} of the input boundary matrix $\mat{D}$ has become clear. The cycles and chains recovered from this decomposition found many uses in attributing topological features to the input data~\cite{dSMVJ11,CSLV22,NiMo24}. However, it is not immediately clear from the algorithm in \cite{milosavljevic2011zigzag} how to recover the matrices $\mat{R}$ and $\mat{V}$ necessary in applications.
Rather than deconstructing the algorithm of Milosavljevic et al.~\cite{milosavljevic2011zigzag}, we simplify it for the case of ordinary persistence.
We describe how to compute two different forms of the $\mat{R} = \mat{D}\mat{V}$ decomposition that come up in applications -- lazy and exhaustive reductions. This allows us to recover specific cycle representatives in matrix multiplication time.
As a benefit, due to its simpler setting, the algorithms in this paper are greatly simplified compared to \cite{milosavljevic2011zigzag}.
%
Finally,  using Dey and Hou's fast zigzag construction~\cite{dey2022fast}, which reduces computation of zigzag persistence to ordinary persistence; the algorithms in this paper give another way to compute zigzag persistence in matrix multiplication time.
\section{Preliminaries}
We are concerned with the persistent homology of a filtration,  an increasing sequence of simplicial complexes indexed by the natural numbers $\mathbb{N}$:
\[ \emptyset  =  K_0 \subset K_1 \subset \ldots \subset K_N =K\]
While persistent homology can be defined more generally, from an algorithmic
perspective this is sufficiently general as nearly all considered cases may be
reduced to this setting. We further assume that $| K_i -  K_{i-1}| = 1$, or
equivalently that the filtration is a total order, i.e., each step in the
filtration adds a single cell. If the initial filtration does not satisfy this
requirement, and defines only a partial order, we may always extend it to an
arbitrary compatible total order, e.g., breaking ties by lexicographical ordering and dimension.

\paragraph{Notation}
Throughout, we use the following notation:
\begin{itemize}
\item Matrices are denoted by bold capital letters, e.g., $\mat{M}$, with sub-matrices indexed by square brackets, e.g., the $j$-th column is $\mat{M}[\cdot,j]$, the $i$-th row is $\mat{M}[i, \cdot]$. 
\item Indexing starts at 1 and indices may be indexed by sets, e.g., $A = \{1,\ldots,k\}$, $\mat{M}[\cdot,A]$ refers to the first $k$ columns. Note that the set may not refer to contiguous columns/rows.
\item We consider permutations implicitly, so $(\mat{P}\cdot\mat{M})[A,\cdot]$ refers to the rows indexed \textbf{after} the permutation $\mat{P}$ is applied to matrix $\mat{M}$. For all the algorithms, performing an explicit permutation, then undoing it, does not affect the asymptotic running time. 
\end{itemize}

\subsection{Persistence Algorithm}

For completeness, we recount the standard persistence algorithm, see \cite{edelsbrunner2002topological}, \cite{zomorodian2004computing}, and \cite{edelsbrunner2022computational} for a more complete description. Let $\bdr$ denote the boundary operator, which we represent as an $\matsize{n}{n}$-matrix over some fixed (finite) field $\field$. 
The rows (top-down) and columns (left-right) are ordered by appearance in the filtration. Hence, the $i$-th row and column represents the simplex $\sigma = K_i - K_{i-1}$  and $\bdr[\cdot,i]$ represents the boundary of $\sigma$ respectively. 
The goal is to compute the $\mat{RV}$-decomposition \cite{cohen2006vines}. That
is, find matrices $\mat{R}$ and $\mat{V}$ such that 
\[\mat{R} = \bdr\cdot \mat{V},\]
where $\mat{R}$ is a matrix in reduced column echelon form and $\mat{V}$ is a full-rank
upper-triangular matrix; we let $\mat{U} = \mat{V}^{-1}$.
We introduce the helper function 
\[\low(j) = 
\argmax\limits_{i} \mat{R}[i,j]\neq 0,  \]
which is defined whenever $\mat{R}[\cdot,j] \neq \mat{0}$.  
When defined, it returns the largest row index $i$ where the column is non-zero; we refer to these lowest
non-zero elements as \textbf{pivots}. Throughout this paper, $\low$ always applies to $\mat{R}$.  The persistence diagram is then given by:
\begin{equation}\label{eq:diagram}    
\dgm =  \left\{(i,j) \mid
    i = \low(j) \right\}
    \; \cup \;
    \left\{
        (i, \infty) \mid
            \mat{R}[\cdot,i]= \mat{0} \;\mbox{and}\;
            i \neq \low(j) \;\forall j
    \right\}.
    \end{equation}


\begin{algorithm}[h!]
\caption{\label{alg:standard}Standard Persistence Algorithm (Lazy reduction)}
\begin{algorithmic}[1]
	\State $\mat{R} = \bdr$
 	\State $\mat{V} = \mat{I}_{n}, \mat{U} = \mat{I}_{n}$
	\For{$j = 1$  to $n$}
	\While{$\mat{R}[\cdot,j]\neq \mat{0}$ and $\exists j' <j$ such that $\low(j) = \low(j')$ }
	\State $i \gets \low(j)$
    \State $\alpha \gets {\mat{R}[i,j]} / {\mat{R}[i,j']}$
	\State $\mat{R}[\cdot,j] \gets \mat{R}[\cdot,j] - \alpha \cdot \mat{R}[\cdot,j']$
    \State $\mat{V}[\cdot,j] \gets \mat{V}[\cdot,j] - \alpha \cdot \mat{V}[\cdot,j']$
    \State $\mat{U}[j', \cdot] \gets \mat{U}[j',\cdot] + \alpha \cdot \mat{U}[j, \cdot]$
                \qquad (equivalently, $\mat{U}[j',j] \gets \alpha$)
	\EndWhile
	\EndFor
\end{algorithmic}
\end{algorithm}

The standard algorithm is a variant of Gaussian elimination.
\cref{alg:standard}, which we call the \textbf{lazy reduction}, reduces each column by considering pivots in
the columns to the left. In each step of the loop, it removes the lowest
non-zero entry until a new pivot is found or the entire column is zeroed out.
A straightforward analysis gives a running time
bound of $O(n^3)$.  In \cite{morozov2005persistence}, a filtration where this algorithm takes cubic
time was given,  showing that the bound is tight.
The updates of matrix $\mat{V}$ follow those of matrix $\mat{R}$. The updates in
matrix $\mat{U}$ undo the updates in $\mat{V}$ to maintain
$\mat{U} \cdot \mat{V} = \mat{I}$. Because the columns of $\mat{R}$ are
processed from left to right, during the update of matrix $\mat{U}$, the row
$\mat{U}[j,\cdot]$ has a single (diagonal) element. Therefore, the update in
$\mat{U}$ is equivalent to just setting the entry $\mat{U}[j',j]$ to $\alpha$.

\begin{algorithm}
\caption{\label{alg:lookahead}``Look-ahead'' Variant of Standard Persistence (Exhaustive reduction)}
\begin{algorithmic}[1]
    \State $\mat{R} = \bdr$
 	\State $\mat{V} = \mat{I}_{n}, \mat{U} = \mat{I}_{n}$
	\For{$j = 1$  to $n$}
	   \If{$\low (j)$ is defined}
	    \State $i \gets \low(j)$
            \For{$j'>j$ and $\mat{R}[\low(j),j']\neq 0$}
                \State $\alpha \gets {\mat{R}[i,j']} / {\mat{R}[i,j]}$
                \State $\mat{R}[\cdot,j']\gets \mat{R}[\cdot,j'] - \alpha \cdot \mat{R}[\cdot,j]$
                \State $\mat{V}[\cdot,j'] \gets \mat{V}[\cdot,j']- \alpha \cdot \mat{V}[\cdot,j]$
                \State $\mat{U}[j, \cdot] \gets \mat{U}[j,\cdot] + \alpha \cdot \mat{U}[j', \cdot]$
                            \qquad (equivalently, $\mat{U}[j,j'] \gets \alpha$!) 
            \EndFor
	   \EndIf
        \EndFor
\end{algorithmic}
\end{algorithm}

We introduce another variant of persistence computation used in applications.  This one ``looks ahead'' and eliminates as many elements from the matrix as it can; so we call it the \textbf{exhaustive reduction}.
In \cref{alg:lookahead}, by the time a column is considered, it is already reduced.
It is then applied to all columns to the right, zeroing out the entire row.
In this way,
when a column is processed, all previous pivots have already been applied.
%

\paragraph{\textbf{Lazy vs. Exhaustive}}

The algorithms based on the two reductions are in a sense the extreme cases: lazy reduction reduces columns only when necessary; and the exhaustive reduction reduces them as much as possible. To illustrate the difference, in the exhaustive reduction, the entire row to the right of a pivot will be zeroed out, whereas in the case of the lazy reduction, only those columns will be zeroed out which have a conflicting pivot in the same row. There are, of course, any number of possible other reductions between these -- but we do not know of any applications that rely on them.


\begin{remark}
As noted in \cite[Section 3.4]{de2011dualities}, persistent cohomology can be computed by performing the same algorithm on the anti-transpose of $\mat{D}$. As such all the algorithms in this paper can compute persistent cohomology.
 \end{remark}

\subsection{Cycle Representatives}

Given an $\mat{R} = \mat{D} \mat{V}$ decomposition,
we distinguish between two types of simplices in the filtration: \emph{positive} simplices create new homology classes; \emph{negative} simplices destroy them.
From the matrix decomposition, one can recover a set of cycle representatives of persistent homology.
There are two sources of this information, although with different content.

\begin{enumerate}
    \item $\mat{V}$ matrix: The columns of positive simplices in $\mat{V}$ correspond to zero columns in the reduced matrix $\mat{R}$. These columns in $\mat{V}$ are cycles, by definition.
\item $\mat{R}$ matrix: The columns of negative simplices in $\mat{R}$ are non-zero. They store linear combinations of boundaries, which are cycles (at the time of their birth) that eventually die in the filtration.
\end{enumerate}

The cycles one recovers from matrix $\mat{R}$ differ between the lazy and exhaustive reductions. Exhaustive reduction was called total reduction by Cohen-Steiner et al.~\cite{CSLV22}. They show that these cycles form a lexicographically optimal basis, which can be used to triangulate point cloud data. Nigmetov and Morozov~\cite{NiMo24} use the lazy reduction -- specifically, the columns and rows from matrices $\mat{V}$ and $\mat{U}$ -- to recover ``critical sets,'' subsets of simplicial complexes affected by optimization with the loss formulated in terms of a persistence diagram.

\paragraph{\textbf{Minimum Spanning Acycle Basis}}
The negative simplices in the filtration represent a minimal spanning acycle (MSA) ~\cite{kalai1983enumeration,skraba2020randomly}. The columns in $\mat{V}$ obtained from either lazy or exhaustive reduction -- or any  reduction that never adds columns of positive simplices to  other columns -- have the form $\sum_{\tau\in MSA} \lambda_\tau \tau + \lambda_\sigma \sigma$. The support of the chain is one (positive) simplex $\sigma$ plus a linear combination of simplices in the minimum spanning acycle.

\begin{lemma}
The representatives from $\mat{V}$ are the same for lazy and exhaustive reduction.
\end{lemma}
\begin{proof}
    Which columns of $\mat{R}$ are non-zero does not depend on the algorithm used; the set of negative simplices only depends on the order of the filtration. We observe that the boundaries of the negative simplices form a basis (since their pivots are in distinct rows). By definition, the boundary of a positive simplex may be expressed as a linear combination of boundaries of negative simplices.
Since the boundaries of the negative simplices form a basis, this linear combination is unique. 
\end{proof}

\paragraph{\textbf{Death Basis}} An alternative is to get cycle representatives directly from the $\mat{R}$ matrix. Its non-zero columns represent cycles because multiplying $\mat{D}\cdot \mat{R} = 0$. This follows directly since $\mat{R}= \mat{D}\cdot\mat{V} $, so  
$\mat{D}\cdot \mat{R} =  \mat{D}\cdot \mat{D}\cdot\mat{V} = 0,$
since $\mat{D}\cdot \mat{D} =0.$
These columns give cycle representatives for all finite classes, i.e., those homology classes that eventually die.

\begin{remark} Following \cite{de2011dualities},  this may be used for cohomology -- applying the algorithms to $\mat{D}^{\perp
}$, the anti-transpose of $\mat{D}$ to obtain the decomposition $\mat{R}^{\perp
} = \mat{D}^{\perp}\cdot \mat{V}^{\perp
} $. The cocycle representatives are then given by columns in  $\mat{V}^{\perp}$. To see that these are cocycles, observe that by construction, they map to zero via the coboundary operator $\mat{D}^{\perp
    }$ before the corresponding death time, i.e., below the pivot in the corresponding column of $\mat{R}^{\perp}$, as we have reversed the indexing with the anti-transpose. Likewise, the cocycle representatives of essential cocycles are given by the columns of $\mat{V}^{\perp}$ whose: corresponding columns in $\mat{R}^{\perp}$ are zero \emph{and} they must not be coboundaries, i.e., the corresponding rows in $\mat{R}^{\perp}$ must not contain pivots.
\end{remark}

\section{Matrix Preliminaries}
We recount a few classical results on matrix multiplication.  We assume the
matrices are over some field $\field$.

\begin{lemma}\label{lem:block_multiplication}
	Let $\mat{B}$ be an $\matsize{n}{k}$ matrix and $\mat{C}$ be a $\matsize{k}{k}$  matrix. The product $\mat{B}\cdot \mat{C}$ can be computed in $O(nk^{\matexp-1})$-time.
\end{lemma}
\begin{proof}
	Divide $\mat{B}$ into $\lceil \frac{n}{k}\rceil$ sub-matrices $\mat{B}_i$ of size $\matsize{k}{k}$:
 \[ \mat{B}\cdot \mat{C} = \begin{bmatrix} 
\mat{B}_1 \\ \vdots \\ \mat{B}_{\lceil n/k \rceil} 
 \end{bmatrix} \cdot \mat{C} = \begin{bmatrix} 
 \mat{B}_1\cdot \mat{C} \\ \mat{B}_2\cdot \mat{C} \\ \vdots \\ \mat{B}_{\lceil n/k \rceil} \cdot \mat{C}
 \end{bmatrix}  \]
	
	Each $\mat{B}_i \cdot \mat{C}$ is a product of two $\matsize{k}{k}$ matrices, which by definition takes $O(k^{\matexp})$ time.  There are $\lceil\frac{n}{k}\rceil$ products to compute yielding a running time of $O\left(\frac{n+1}{k} k^\matexp\right ) = O(n k^{\matexp-1})$.
\end{proof}

\paragraph{\textbf{Column Operations via Matrix Multiplication}}
    Here we relate matrix multiplication with the reduction steps used in the persistence algorithms, i.e.,  representing column operations via matrix multiplication.  Assume that we have the following block matrix,
\[
\left[
    \begin{array}{c|c}\mat{B} &\mat{C}\end{array} \right] = \left[\begin{array}{ccc|ccc}b(1) &\ldots& b(n) &c(1) &\ldots & c(m) \end{array}\right]
\]
where $b(i)$ and $c(j)$ are (column) vectors. Reducing $c(j)$ which we denote $c'(j)$, using the vectors $b(i)$, is expressed as a linear combination
\[c'(j) = c(j) + \sum_i \lambda_i(j) b(i),\]
where $\lambda_i(j)$ are the coefficients. 
Reducing only the $j$-th column of $\mat{C}$ may be written as
\[\begin{bmatrix}\mat{B} & c'(j) \end{bmatrix} = 
\begin{bmatrix}\mat{B} &  c(j) \end{bmatrix}  \cdot  \begin{bmatrix}\mat{I} & \lambda(j) \\ 0 & \mat{1}_j  \end{bmatrix},
\]
where $\mat{1}_j$ is a column vector where the $j$-th entry is 1 and zero  everywhere else. If we consider
$\mat{\Lambda}[i,j] := 
\lambda_i(j),$
we can rewrite the reduction of $\mat{C}$ by
\[\begin{bmatrix}\mat{B} & \mat{C}' \end{bmatrix} = 
\begin{bmatrix}\mat{B} &  \mat{C} \end{bmatrix}  \cdot  \begin{bmatrix}\mat{I} & \mat{\Lambda}\\ 0 & \mat{I}  \end{bmatrix},
\]
or equivalently, the reduced matrix is 
$\mat{C}' =  \mat{C} - \mat{B} \cdot \mat{\Lambda}.$
\begin{lemma}\label{lem:perm}
    Multiplying a permutation matrix with an $\matsize{n}{m}$ matrix, takes $O(nm)$ time.
\end{lemma}
\begin{proof}
As a permutation matrix rewrites each element once, the total number of elements gives the bound. 
\end{proof}

\section{Column Algorithms}
Here we describe the algorithm in \cite{milosavljevic2011zigzag} for the special case of classical persistent homology. To aid in exposition, we first recall the \textbf{Schur complement}. Given a block matrix
\[
\begin{bmatrix}
    \mat{B}_1 & \mat{C}_1 \\
    \mat{B}_2 & \mat{C}_2
\end{bmatrix},\]
assuming $\mat{B}_2$ is non-singular, we can compute an updated matrix where the rows corresponding to $\mat{B}_2$ are zeroed out by the following transformation:

\begin{equation}\label{eq:schur}
	\begin{bmatrix}
        \mat{B}_1 & \mat{C}_1 \\
        \mat{B}_2 & \mat{C}_2
    \end{bmatrix}\begin{bmatrix} \mat{I}& -\mat{B}_2^{-1} \mat{C}_2 \\ 0 & \mat{I} \end{bmatrix} =  \begin{bmatrix} \mat{B}_1 & \mat{C}_1 - \mat{B}_1 \mat{B}^{-1}_2 \mat{C}_2 \\ \mat{B}_2& 0 \end{bmatrix}
\end{equation}
where $\mat{C}_1 - \mat{B}_1 \mat{B}^{-1}_2 \mat{C}_2$ is referred to as the Schur complement.

\begin{observation}
The matrix $\mat{B}_2^{-1}\mat{C}_2$ encodes the column operations required to zero out $\mat{C_2}$, that is, $\mat{B}_2^{-1}\mat{C}_2=\mat{\Lambda}$.
\end{observation}

A key fact we will use in the algorithm is that matrix inversion has the same algorithmic complexity as matrix multiplication. 
\begin{restatable}{theorem}{thmmatrixinv}
    \label{thm:matrix_inv}
    Inverting a (square) non-singular  (upper or lower) triangular matrix can be done in matrix multiplication time.
\end{restatable}
\begin{proof}
See \Cref{appendix:fastinversion}
for a  simplified proof for the special case of triangular matrices.
\end{proof}

\subsection{Fast Exhaustive Algorithm}

\begin{algorithm}[h]\caption{Column Algorithm($A = [i,j]$)}
\label{alg:col-tree-no-inversion}
\begin{algorithmic}[1]
   \If{$i = j$} \algorithmiccomment{leaf}
        \If{$\mat{R}[\cdot, i] = \mat{0}$} \Comment{If no pivot exists}
            \label{line:zero-column}
            \State $\ell \gets  \arg\min_{\ell>n} \mat{R}[\cdot, \ell] \neq 0$ \Comment{find the next pivot in $\mat{I}_{m+n}$} 
            \State Permute columns $i$ and $\ell$
            \State $Z\gets Z \cup i$ 
        \EndIf
    \Else 
        \State $B \gets [i, (i+j)/2)$ \Comment{left child} 
        \State $C \gets [(i+j)/2, j]$ \Comment{right child}
        \State Recurse on $B$
        \State Apply updates from columns $B$ to $C$: \tikz[remember picture] \node (update) {};
        \State \qquad $L \gets \langle \low(\ell) : \ell \in B \rangle$ \Comment{Implicit permutation; $L$ preserves order}
        \State \qquad $\mat{\Lambda} \gets \mat{R}[L,B]^{-1} \cdot \mat{R}[L,C]$ \Comment{$\mat{A}_2^{-1} \mat{B}_2$} \label{line:update}
        \State \qquad $\mat{R}[\bar{L},C] \gets \mat{R}[\bar{L},C] - \mat{R}[\bar{L},B]\cdot \mat{\Lambda}$ \Comment{Apply $B$ to $C$; Schur update $(k \times k) \cdot (k \times n)$} 
        \State \qquad $\mat{R}[L,C] =0$
        \State \qquad $\mat{V}[\cdot,C] \gets \mat{V}[\cdot,C] -\mat{V}[\cdot,B]\cdot \mat{\Lambda}$ \Comment{Matching update in $\mat{V}$}
        \State Recurse on $C$ 
    \EndIf
\end{algorithmic}
\end{algorithm}

We present a fast version of exhaustive reduction in \cref{alg:col-tree-no-inversion}.
The general structure follows \cref{alg:lookahead}. However, rather than apply
the pivots when we find them, we apply them in batches via the Schur complement through a standard recursion.
We first augment the boundary matrix with an identity matrix, which is of size at least $n$, for reasons we shall describe below. To simplify notation, we assume $n$ is a power of $2$. If it is not, we can set the size of the identity matrix to $m+n$ with $m<n$ such that $m+2n = 2^x$. Therefore, our initial matrices are:
\[\mat{R} = \begin{bmatrix} \bdr & 0  \\ 0 & \mat{I}_{m+n} \end{bmatrix}, \qquad\qquad \mat{V} = \mat{I}_{m+2n}.\]
Throughout the algorithm, we implicitly perform two types of permutations by
taking appropriate subsets of rows and columns.
The first are column permutations.  To ensure that the appropriate submatrices are invertible, we require
that after processing $k$ columns of $\mat{R}$, there are $k$ pivots, or, equivalently,
that the matrix corresponding to the processed columns is full (column) rank.
Because cycles reduce to zero columns, the if-statement in
\cref{line:zero-column} permutes a zero column from $\mat{R}$
with a non-zero column from the identity matrix $\mat{I}_{m+n}$, and records the
indices of such columns in list $Z$, to eventually undo the permutation.

\begin{figure}[htbp]
\begin{center}
 \begin{tikzpicture}
        \node[anchor=south west, inner sep=0] (image) at (0,0) {\includegraphics[width=\textwidth,page=16]{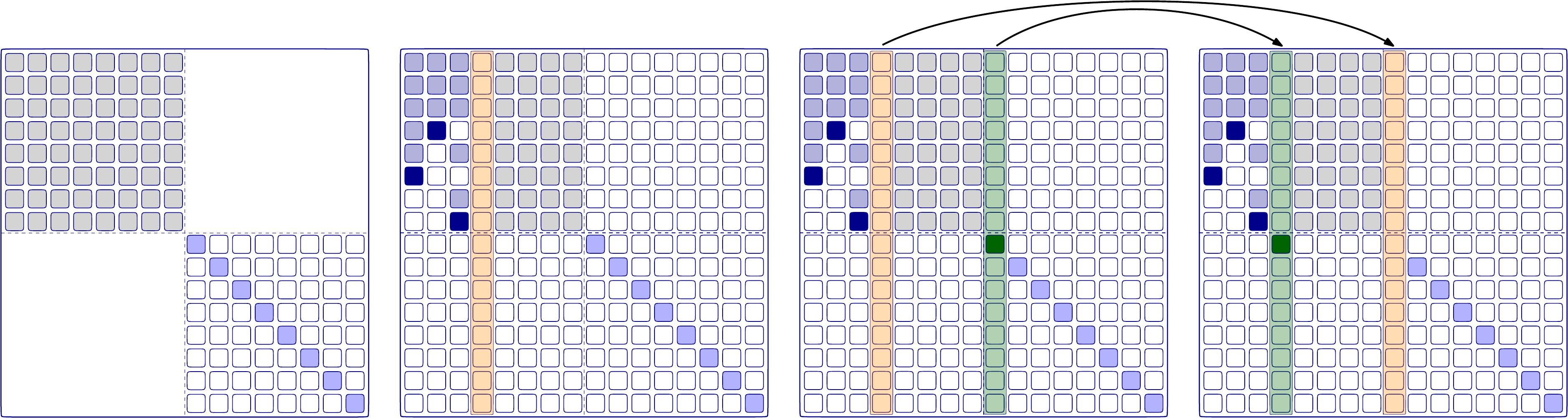}}; %
        \begin{scope}[x={(image.south east)}, y={(image.north west)}] 
            \node[anchor=center] (s1)  at (0.12,-0.1) {(a)};
            \node[anchor=center] (s2)  at (0.37,-0.1) {(b)};
            \node[anchor=center] (s3)  at (0.627,-0.1) {(c)};
            \node[anchor=center] (s5)  at (0.882,-0.1) {(d)};
            \draw [thick, decorate, decoration = {calligraphic brace}] (0.122,0.88) --  (0.122,0.45);
            \node (submatrix1)  at (0.095,0.22) {\large{$\mat{I}$}};
            \node (submatrix2)  at (0.147,0.66) {\large{$\mat{D}$}};
            \draw [thick, decorate, decoration = {calligraphic brace}] (0.115,0.01) --  (0.115, 0.435);
           \draw [thick, decorate, decoration = {calligraphic brace}] (0.255,0.9) --  (0.315, 0.9);
           \draw [thick, decorate, decoration = {calligraphic brace}] (0.316,0.9) --  (0.373, 0.9);
           \node (submatrix3)  at (0.286,0.975) {$k$};
           \node (submatrix3)  at (0.346,0.975) {$k$};
        \end{scope}
     \end{tikzpicture}
\caption{\label{fig:col_permute} An example of the column permutation in the exhaustive algorithm. (a) Initially, we append an identity matrix to the boundary matrix. (b) Assume we have processed the first $k$ columns. The blue entries represent non-zero entries, and the dark blue boxes represent the pivots. To apply these pivots to next $k$ columns (shown in gray), the first $k$ columns must contain $k$ pivots. (c-d) If there is a  a zero column (shown in orange),  we transpose it with the first non-zero entry in the appended identity matrix, ensuring the processed $k$ columns all have pivots. }
\end{center}
\end{figure}

We must also consider row permutations.
To apply \cref{thm:matrix_inv}, we must find the  inverse $\mat{R}^{-1}[L,B]$ in
\cref{line:update}, which is used in the Schur complement. We construct a sequence
$L = \langle \low(\ell) : \ell \in B \rangle$ of the
pivot rows in columns $B$. Denoting by $\bar{L}$ the sequence of
rows outside of $L$, we denote a permutation that places rows $L$ below $\bar{L}$
as $\mat{P}_L$:
\begin{equation}\label{eq:column_alg_row_permute}
    \left( \mat{P}_L \cdot \mat{R} \right) =
        \begin{bmatrix}
            \mat{R}[\bar{L}, \cdot] \\
            \mat{R}[L,\cdot]
        \end{bmatrix}.
\end{equation}
See \cref{fig:row_permute}(b-c). The permutation is explicitly given by, 
\[P_L: \low (B(i))\mapsto m+n-|B|+i \qquad \mbox{for}\;\; i=1,\ldots,|B|\]
where $B(i)$ is the $i$-th entry of the indices in $B$.

\begin{figure}[htbp]
\begin{center}
\begin{tikzpicture}
        \node[anchor=south west, inner sep=0] (image) at (0,0) {\includegraphics[width=\textwidth,page=17]{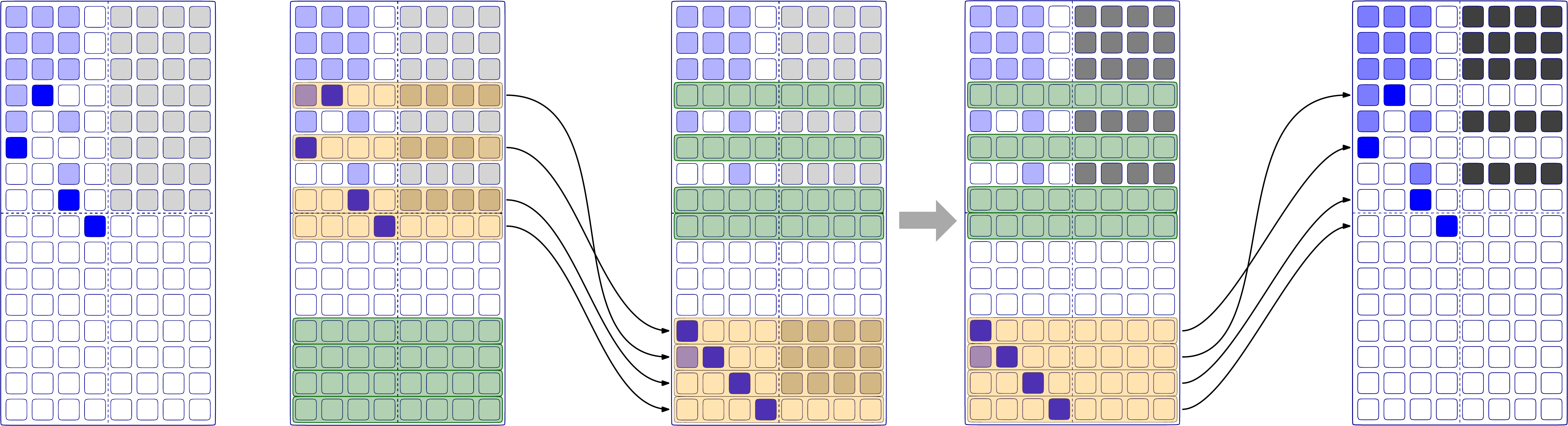}}; %
        \begin{scope}[x={(image.south east)}, y={(image.north west)}] 
            \node[anchor=center] (s1)  at (0.07,-0.1) {(a)};
            \node[anchor=center] (s2)  at (0.254,-0.1) {(b)};
            \node[anchor=center] (s3)  at (0.496,-0.1) {(c)};
            \node[anchor=center] (s4)  at (0.684,-0.1) {(d)};
            \node[anchor=center] (s5)  at (0.931,-0.1) {(e)};

        \end{scope}
     \end{tikzpicture}
\caption{\label{fig:row_permute} An example of the row permutations. (a) A submatrix consisting of 8 columns where the 4 columns on the left have been reduced, and we need to apply the pivots to the following 4 columns (in gray). (b-c) Each row which has a pivot gets permuted to the bottom of the matrix in the same order as the columns -- resulting in a lower triangular matrix. (d) Applying the Schur complement zeros out the gray entries on the bottom right and we update the entries in the upper right (dark gray). (e) With the pivot rows zeroed out on the right, we reverse the  permutation.  }
\end{center}
\end{figure}

\begin{lemma}\label{lem:triangular}
    After permuting the rows by $\mat{P}_L$, the bottom-left $\matsize{k}{k}$ matrix is lower triangular, where $k = |L|$.
\end{lemma}
We delay the proof until after the description of the algorithm.  We now follow a recursion: we split matrix $\mat{R}$ into  the
first half of the columns denoted by $B$ and the second half denoted by $C$. We proceed to call the
algorithm on the first half of the matrix. Once we reach a leaf in the
recursion tree (shown in Figure \ref{fig:recursion}), the matrix consists of a single column. The key invariant is that whenever we
reach a leaf, the corresponding column is reduced with respect to all the columns which come prior to
it in the filtration. For the first column, this is tautologically true and we delay the proof for the general case until later.
%

\begin{figure}[t]
\centering\includegraphics[page=27,width=0.4\textwidth]{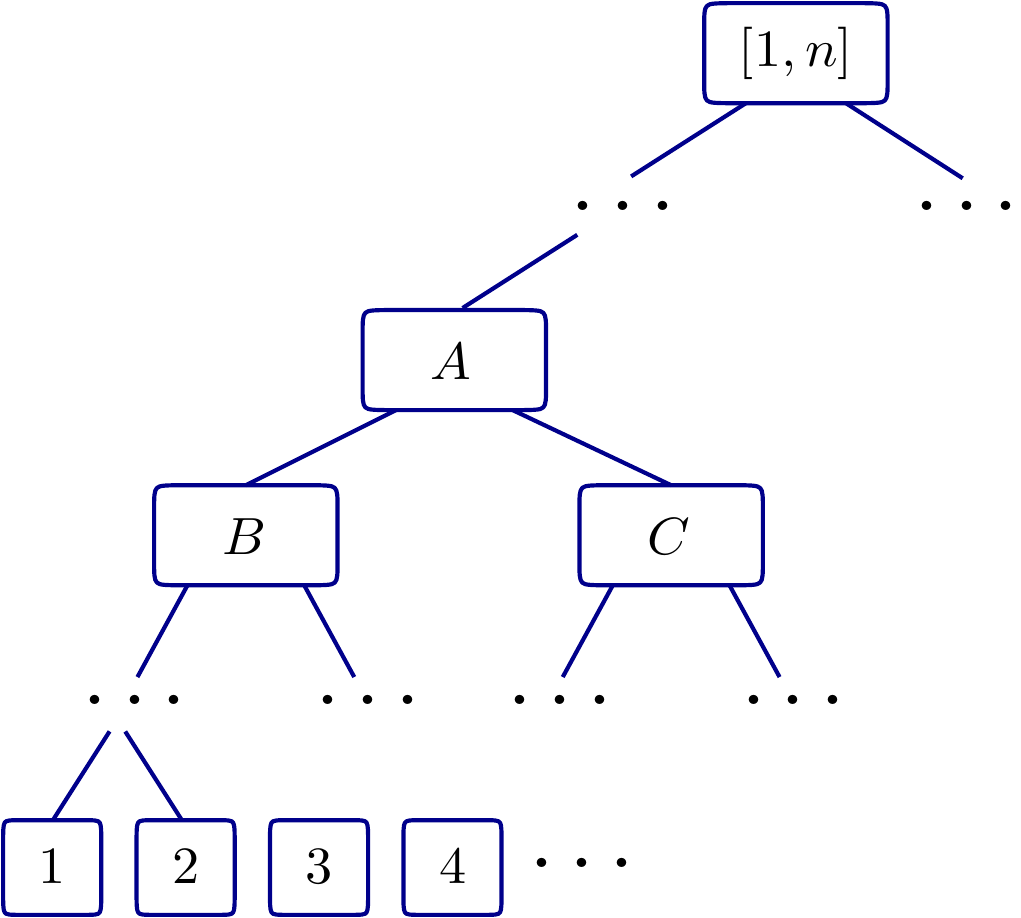}
\caption{\label{fig:recursion} The standard binary recursion tree in how we split the matrix by column. Observe that $|B| = |C|=|L|$ in \cref{alg:col-tree-no-inversion} and that $\mat{R}[L,B] = (\mat{P}_L \cdot \mat{R})[\{n-|B|+1,\ldots, n\},B]$, the bottom $|B|$ rows in the permuted matrix.     }
\end{figure}
If we are not in the base case (a single column), we must apply the pivots in the columns indexed by $B$ (left half of the submatrix) to the the columns indexed by $C$ (right half of the submatrix). As described in \cref{eq:column_alg_row_permute}, we construct the permutation
 $\mat{P}_L$ such that
\[
    (\mat{P}_L \cdot \mat{R})[\cdot, B\cup C]) =
        \begin{bmatrix}
            \mat{R}[\bar{L}, B] &  \mat{R}[\bar{L},C] \\
            \mat{R}[L,B]        &  \mat{R}[L,C]
        \end{bmatrix}
\]
By construction of $\mat{P}_L$, $\mat{R}[L,B]$ is lower-triangular with
non-zero diagonal entries and so is full-rank.
We can then apply the Schur complement to zero out all the corresponding rows in
the columns $C$,
$\mat{R}[L,C] \gets 0$ and update the rest of the rows:
\[
    \underbrace{\mat{R}[\bar{L},C]}_{\mat{C}_1\vphantom{\mat{B}_2^{-1}}}\gets
    \underbrace{\mat{R}[\bar{L},C]}_{\mat{C}_1\vphantom{\mat{B}_2^{-1}}}
        - \underbrace{\mat{R}[\bar{L},B]}_{\mat{B}_1\vphantom{\mat{B}_2^{-1}}} \cdot
          \underbrace{\mat{R}[L,B]^{-1}}_{\mat{B}_2^{-1}} \cdot \underbrace{\mat{R}[L,C]}_{\mat{C}_2\vphantom{\mat{B}_2^{-1}}}
    =
    \mat{R}[\bar{L},C] - \mat{R}[\bar{L},B] \cdot \mat{\Lambda}
\]

\noindent To update matrix $\mat{V}$, we perform the same operations
$\mat{\Lambda} = \mat{R}[L,B]^{-1} \cdot \mat{R}[L,C]$ on $\mat{V}$:
\[
    \mat{V}[\bar{L},C] \gets \mat{V}[\bar{L},C] - \mat{V}[\bar{L},B] \cdot \mat{\Lambda}
\]

Once the left half of the matrix has been applied to the right half, we
recurse on the second half, i.e., $C$. Here again,
because $\mat{R}[L,B]$ is lower-triangular full-rank and $\mat{R}[L,C] = 0$,
the first column of $\mat{R}[\cdot,C]$ is reduced.
We now prove correctness. 

\begin{lemma}\label{lem:recurse}
When the recursion base case is reached for column $k$, $\mat{R}[\low(j),k]=0$ for all $j<k$.
\end{lemma}
\begin{proof}
For $k=1$, this is tautological since $\low  (j)$ is empty. For $k>1$, we observe that entry $\mat{R}[\low  (i),k]$ was made zero when $i\in B$ and $k\in C$. As the recursion always splits the column range in half, this condition must be satisfied somewhere in the recursion before the base case of $k$ is reached.
\end{proof}

\begin{proof}[Proof of Lemma~\ref{lem:triangular}]
    We observe that when $\mat{P}_L$ is applied, the columns corresponding to $L$ have a pivot.  By construction,  $\mat{P}_L$ places the pivots onto the diagonal in the bottom $|L|$ rows. For any of these rows above the diagonal, the entries are 0 by Lemma~\ref{lem:recurse}, implying the bottom $k\times k$ matrix is lower triangular. 
\end{proof}
\begin{theorem}
Algorithm \ref{alg:col-tree-no-inversion} computes the persistence diagram correctly.
\end{theorem}
\begin{proof}
By \cref{eq:diagram}, it suffices to show that for all $i$, $\low(i)$ is computed correctly. For $i=1$, this is trivially true. For $i>1$, observe that at any step of the algorithm, we maintain 
\[\mat{R} = \mat{D}\cdot \mat{V},\]
as any operations on $\mat{R}$ are applied to $\mat{V}$.  
Assuming the pivots are correct for $j<i$, by the above and Lemma \ref{lem:recurse}, $\mat{V}$ gives a transformation depending only on columns $j<i$, such that all elements in $\mat{R}[\low (j),i]=0$. If $\low (i)$ does not exist, $D[\cdot,i]$ was in the span of the previous columns and hence a cycle.  If $\low (i)$ exists, it  is distinct from $\{\low (j)\}_{j<i}$ and so we have a new pairing. The correctness of this pairing follows from the Pairing Uniqueness Lemma~\cite{cohen2006vines}. 
\end{proof}
\paragraph{\textbf{Running Time Analysis}}
The base case takes $O(1)$ time. For a general step where $|B|=|C|=k$,  applying
the update requires a row permutation of a $\matsize{2n}{2k}$ matrix which takes
$O(nk)$ time. We must compute $\mat{R}[L,B]^{-1}$ which is inverting a $\matsize{k}{k}$ matrix.   Multiplying it with $\mat{R}[L,C]$  takes $O(k^{\matexp})$ time as both are $\matsize{k}{k}$ matrices. Finally, we 
multiply this product with the $\matsize{n}{k}$ matrix $\mat{R}[\bar{L},C]$, which by Lemma~\ref{lem:block_multiplication}, takes
$O(nk^{\matexp-1})$ time.
Solving
the recursion, we find that the total running time is  $O(n^{\matexp})$ for
$\matexp>2$ or $O(n^2\log n)$ if $\matexp=2$. 
See \Cref{appendix:recursion-analysis} for the complete
derivation.

\section{Row Algorithm}
While \cref{alg:col-tree-no-inversion} produces the exhaustive reduction in matrix multiplication time, a natural question is whether the same can be done for  computing the lazy reduction (and its representatives).  In this section, we present an algorithm which achieves this. We first give the iterative version before describing how it can be computed in $O(n^{\matexp})$. 
\begin{algorithm}[h]\caption{Incremental Row Algorithm}
\label{alg:row-incremental}
\begin{algorithmic}[1]
\State $\mat{R} = \bdr$
 	\State $\mat{V} = \mat{I}_{n}, \mat{U} = \mat{I}_{n}$
	\For{$i = n$  to $1$}
	   \If{$\lft (i)$ is defined}
	    \State $j \gets \lft(i)$
            \For{$j'>j$ and $\low(j')= i$ }
                \State $\alpha \gets {\mat{R}[i,j']} / {\mat{R}[i,j]}$
                \State $\mat{R}[\cdot,j'] \gets \mat{R}[\cdot,j'] - \alpha \cdot \mat{R}[\cdot,j]$
                \State $\mat{V}[\cdot,j'] \gets \mat{V}[\cdot,j']- \alpha \cdot \mat{V}[\cdot,j]$
                \State $\mat{U}[j, \cdot] \gets \mat{U}[j,\cdot] + \alpha \cdot \mat{U}[j', \cdot]$
                            \qquad (equivalently, $\mat{U}[j,j'] \gets \alpha$)
            \EndFor
	   \EndIf
        \EndFor
\end{algorithmic}
\end{algorithm}
The idea behind this reduction is to consider rows from the bottom up. Observe that a pivot in the bottom row can be directly identified in $\mat{D}$ as the earliest (leftmost) non-zero entry, similarly as the pivot in the first column can be directly identified. To formalize the notion of the earliest eligible pivot in a given row, we define
\[
    \lft(i) = \argmin\limits_j \left\{ \mat{R}[i,j]\neq 0 :  \low(j) = i \right\} 
\]
As in the case of $\low$, we only apply this function to $\mat{R}$. The function $\lft $ returns the earliest column $j$ which is non-zero and is the lowest such non-zero entry in its column.  This condition is important as the earliest (left-most) non-zero entry may have pivots below it. This column is used to zero out the $i$-th row in columns to the right provided they do not already contain a pivot, i.e., if $\low(j)=i$, then the column is applied.  This condition is why we obtain equivalence with the lazy reduction --  we do not apply the pivot column to later columns which already contain a pivot. 
\begin{lemma}
    Algorithm~\ref{alg:row-incremental}  produces the same decomposition as the lazy reduction.
\end{lemma}
\begin{proof}
This is equivalent to \cite[Theorem 3.1]{de2011dualities}.
\end{proof}

\subsection{Fast Row Algorithm}

The incremental row algorithm (\cref{alg:row-incremental}) takes $O(n^3)$ time in the worst case because, at each row, subtracting the pivot column from all eligible columns to its right can touch the entire matrix above that row.
To achieve $O(n^\matexp)$ time, \cref{alg:row-tree} replaces these immediate column-wide updates with a divide-and-conquer recursion that applies them in batches via matrix multiplication.
Unlike the column algorithm (\cref{alg:col-tree-no-inversion}), this approach does not require matrix inversion and needs less padding; the trade-off is a two-stage batch update to propagate delayed column operations upward through the rows.

\paragraph{Setup and padding}
Assuming $\bdr$ is an $\matsize{n}{n}$ matrix, we choose $m < n$ so that $m + n = 2^x$ is a power of~$2$, pad $\bdr$ at the bottom with $m$ zero rows, and append an identity matrix $\mat{I}_{m+n}$ on the right:
\[
\mat{R} = 
\left[
\begin{array}{l|c}
\bdr & \multirow{2}{*}{$\mat{I}_{m+n}$}\\
\mat{0}_m& 
\end{array}
\right], \qquad\qquad \mat{\Lambda} = \mat{I}_{m+2n}. 
\]
Appending $\mat{I}_{m+n}$ ensures that every row has a non-zero entry with no pivot below it, so $\lft(\cdot)$ is always defined. Because the auxiliary identity columns appear to the right of all original columns in the filtration order, an auxiliary column is only chosen as a pivot when a row of $\bdr$ has no eligible non-zero entry of its own, and it never reduces any original column.
To simplify notation, from this point on we assume $n$ is already a power of~$2$ ($m = 0$), so that $\mat{R} = \begin{bmatrix} \bdr & \mat{I}_n \end{bmatrix}$ has $n$ rows and $2n$ columns, and $\mat{\Lambda} = \mat{I}_{2n}$ records the column operations.

\paragraph{Organizing pivots via column permutations}
In \cref{alg:row-incremental}, a pivot column at row $r$ is only used to eliminate entries in columns $j'$ that satisfy two conditions: (1)~$\low(j') = r$ (column $j'$ has no non-zero entries below row $r$), and (2)~$j'$ appears after the pivot column in the filtration order.
To make this selective elimination easier for block matrix multiplication, we separate columns that already have pivots in lower rows from those that do not by maintaining an implicit column permutation $\mat{P}$.

We process the $n$ rows from bottom to top in steps $i = 1, \ldots, n$, where step $i$ processes row $r = n - i + 1$.
At step $i$, we identify the pivot column $k = \lft(n - i + 1)$ and swap it into column position $i$ via the transposition
\[
    \pi(i): \lft(n - i + 1) \leftrightarrow i.
\]
After $i$ steps, the accumulated column permutation is
\[
    \mat{P}_i = \prod\limits_{j \le i} \pi(j).
\]
Once a pivot column is moved into position $j \le i$, later steps only transpose indices which come after $j$, so the first $i$ columns of the column-permuted matrix remain fixed for the rest of the algorithm.
 For brevity, we omit the subscript and write $\mat{P}$ for the accumulated permutation, keeping permutations implicit via
\[
    \widehat{\mat{R}} = \mat{R} \cdot \mat{P},
    \qquad
    \widehat{\mat{\Lambda}} = \mat{P}^T \cdot \mat{\Lambda} \cdot \mat{P}.
\]
Conjugating $\mat{\Lambda}$ by $\mat{P}$ permutes both its rows and columns to match the permuted column order of $\widehat{\mat{R}}$.

\begin{figure}[t]
\begin{center}
\begin{tikzpicture}
        \node[anchor=south west, inner sep=0] (image) at (0,0) {\includegraphics[width=0.7\textwidth,page=24]{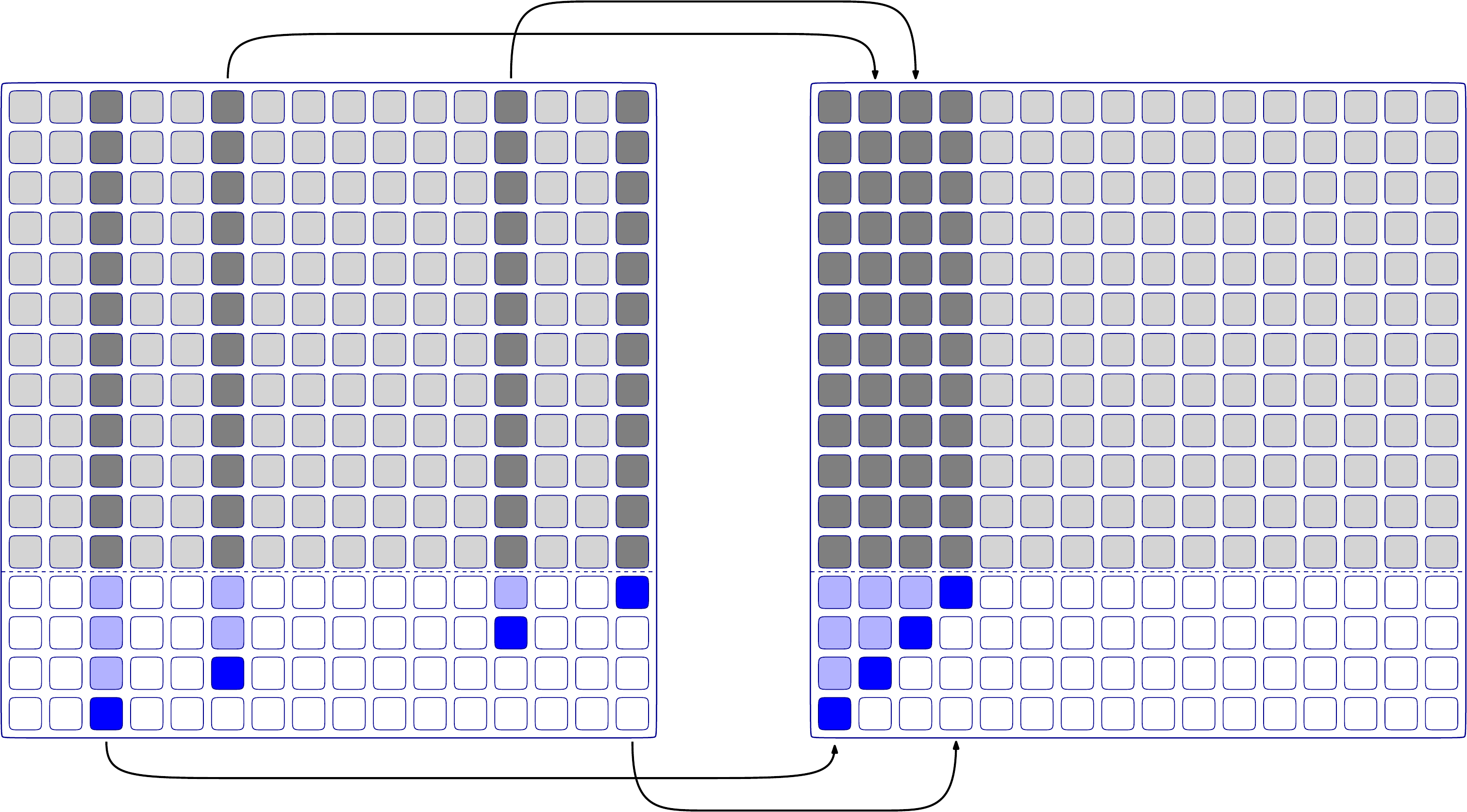}};
        \begin{scope}[x={(image.south east)}, y={(image.north west)}] 
            \node[anchor=center] (s1)  at (0.24,-0.02) {(a)};
            \node[anchor=center] (s2)  at (0.78,-0.02) {(b)};
\end{scope}
\end{tikzpicture}
\caption{\label{fig:row_alg_permutations} Column permutations in the row algorithm. (a)~As rows are processed from the bottom up, each pivot column is identified via $\lft(\cdot)$. (b)~Permuting the $i$-th pivot column to position $i$ places the pivots of the bottom $k$ rows along the anti-diagonal on the left. Every column to the right of position $k$ is zero in the bottom $k$ rows, ensuring that we only eliminate entries that are the lowest non-zeros in their columns.}
\end{center}
\end{figure}

As illustrated in \cref{fig:row_alg_permutations}, after $i - 1$ steps the pivots of the bottom $i - 1$ rows sit along the anti-diagonal in columns $1, \ldots, i - 1$ of $\widehat{\mat{R}}$.
This layout has a crucial consequence at step $i$ (row $n - i + 1$):
\begin{itemize}
    \item Every column with a pivot below row $n - i + 1$ has already been moved to a position $<i$.
    \item Consequently, every remaining column at position $\ge i$ in $\widehat{\mat{R}}$ is completely zero below row $n - i + 1$. Any non-zero entry in row $n - i + 1$ at position $\ge i$ is therefore automatically the lowest non-zero entry ($\low$) in its column.
    \item Among all columns at positions $\ge i$ that are non-zero in row $n - i + 1$, $\lft(n - i + 1)$ selects the one that appeared \emph{earliest in the original filtration order}. Let $k \ge i$ be its current position in $\widehat{\mat{R}}$.
    \item After transposing $i \leftrightarrow k$, the pivot for row $n - i + 1$ sits at $\widehat{\mat{R}}[n - i + 1, i]$, and every non-zero entry to its right ($j>i$) belongs to a column that is both eligible ($\low = n - i + 1$) and later than the pivot column in the original filtration order.
\end{itemize}

\begin{remark}
    \label{rem:pivots}
    Because all columns whose pivots lie below row $n - i + 1$ have been permuted to positions $<i$, every non-zero entry in row $n - i + 1$ at position $j>i$ is a lowest non-zero entry in its column. Clearing all entries to the right of column $i$ in row $n - i + 1$ (\cref{lin:zero-out} of \cref{alg:row-tree}) therefore performs the exact same eliminations as the lazy reduction in \cref{alg:row-incremental}.
\end{remark}

Thus, at the base case for step $i$, once the pivot column is placed at position $i$, we simply record the multipliers needed to clear row $n - i + 1$ to the right of column $i$,
\[
    \widehat{\mat{\Lambda}}[i, j] = (\mat{P}^T \cdot \mat{\Lambda} \cdot \mat{P})[i, j] =
        \begin{cases}
            -\dfrac{(\mat{R}\cdot \mat{P})[n-i+1,j]}{(\mat{R}\cdot \mat{P})[n-i+1,i]} & \text{if } j > i,\\[8pt]
            1 & \text{if } j = i,\\
            0 & \text{if } j < i,
        \end{cases}
\]
and set $(\mat{R} \cdot \mat{P})[n - i + 1, \{i + 1, \ldots, 2n\}] \gets 0$ (since $\mat{\Lambda}$ is initialized to $\mat{I}_{2n}$, the diagonal entry $j = i$ is already $1$).

\begin{algorithm}[t]\caption{Row Algorithm($A = [i,j]$)}
\label{alg:row-tree}
\begin{algorithmic}[1]
    \If{$i = j$} \algorithmiccomment{leaf: process row $n-i+1$}
        \State $k \gets \lft  (n-i+1)$ 
        \Comment{Identify the pivot in row $(n-i+1)$}
        \State Update permutation matrix $\mat{P}$ with transposition $i \leftrightarrow k$.
        \State $(\mat{P}^T \!\cdot \!\mat{\Lambda}\!\cdot\!\mat{P} )[i,\ell] \gets
        -\frac{(\mat{R}\cdot\mat{P} )[n-i+1,\ell]}{(\mat{R}\cdot\mat{P})[n-i+1,i]}$ for $\ell > i$.
        \State $(\mat{R}\!\cdot\!\mat{P})[n-i+1, \{i+1, \ldots,2n\}]\gets 0 $ \label{lin:zero-out} \Comment{Zero out $(n-i+1)$-st row; see \cref{rem:pivots}}
        \State \Return $1$ \Comment{Base case: return $1\times 1$ identity matrix}
     \Else
         \State $B \gets [i, (i+j)/2)$ \Comment{left child (first half of steps/columns)}
         \State $C \gets [(i+j)/2, j]$ \Comment{right child (second half of steps/columns)}
         \State $B_r \gets n-B+1$ \Comment{bottom half of rows in this call}
         \State $C_r \gets n-C+1$ \Comment{top half of rows in this call}
         \State $\overline{B} \gets [(i+j)/2, 2n]$ \Comment{all columns to the right of $B$}
         \State $\mat{A}_1 \gets$ Recurse on $B$
         \State Update pivot columns $B$ in rows $C_r$: \label{line:row-internal-update}
         \Statex \qquad $(\mat{R}\!\cdot\!\mat{P})[C_r, B] \gets
         (\mat{R}\!\cdot\!\mat{P})[C_r, B] \cdot \mat{A}_1$   
                     \Comment{$(k \times k) \cdot (k \times k)$}
         \State Propagate updates from columns $B$ to $\overline{B}$ in rows $C_r$: \label{line:row-external-update} \Comment{$(k \times k) \cdot (k \times 2n)$}
         \Statex \qquad $(\mat{R}\!\cdot\!\mat{P} )[C_r,\overline{B} ] \gets (\mat{R}\!\cdot\!\mat{P} )[C_r, \overline{B}]  +
                     (\mat{R}\!\cdot\!\mat{P})[C_r,B] \cdot (\mat{P}^T\!\cdot\!\mat{\Lambda}\!\cdot\!\mat{P})[B,\overline{B}]$
         \State $\mat{A}_2 \gets$ Recurse on $C$
         \State \Return $\begin{pmatrix}\mat{A}_1 & \mat{A}_1\cdot (\mat{P}^T\!\cdot\!\mat{\Lambda}\!\cdot\!\mat{P})[B,C]\cdot \mat{A}_2 \\\mat{0} &\mat{A}_2\end{pmatrix}$ \label{line:row-return} \Comment{combined transition matrix on $A = B \cup C$}
    \EndIf
\end{algorithmic}
\end{algorithm}

\paragraph{Delayed updates and structure of the recursion}
At a leaf (step $i$), we clear the entries to the right of the pivot in row $n - i + 1$ immediately, but we \emph{do not} immediately apply those column operations to the rows above $n - i + 1$.
Doing so at every leaf would require $O(n^2)$ work per row; instead, we record the multipliers in $\mat{\Lambda}$ and apply them to higher rows in batches as the recursion unwinds.

A recursive call in \cref{alg:row-tree} is parameterized by an interval of step indices $A = [i, j] \subseteq \{1, \ldots, n\}$.
Because steps proceed from $1$ to $n$, the interval $A$ simultaneously represents:
\begin{enumerate}
    \item the target column positions $t \in A$ (ordered left to right) where the pivots found during this call will be placed in $\widehat{\mat{R}}$, and
    \item the block of rows $A_r = n - A + 1 = \{n - t + 1 : t \in A\}$ (ordered bottom to top) processed during this call.
\end{enumerate}
When $A$ is split into its left child $B = [i, (i+j)/2)$ and right child $C = [(i+j)/2, j]$ of size $k = |B| = |C|$, the rows $B_r$ form the \emph{bottom} half of $A_r$ and the rows $C_r$ form the \emph{top} half of $A_r$.
Because rows must be reduced from bottom to top, the algorithm recurses on $B$ first.

\begin{figure}[t]
\begin{center}
\begin{tikzpicture}
        \node[anchor=south west, inner sep=0] (image) at (0,0) {\includegraphics[width=\textwidth,page=29]{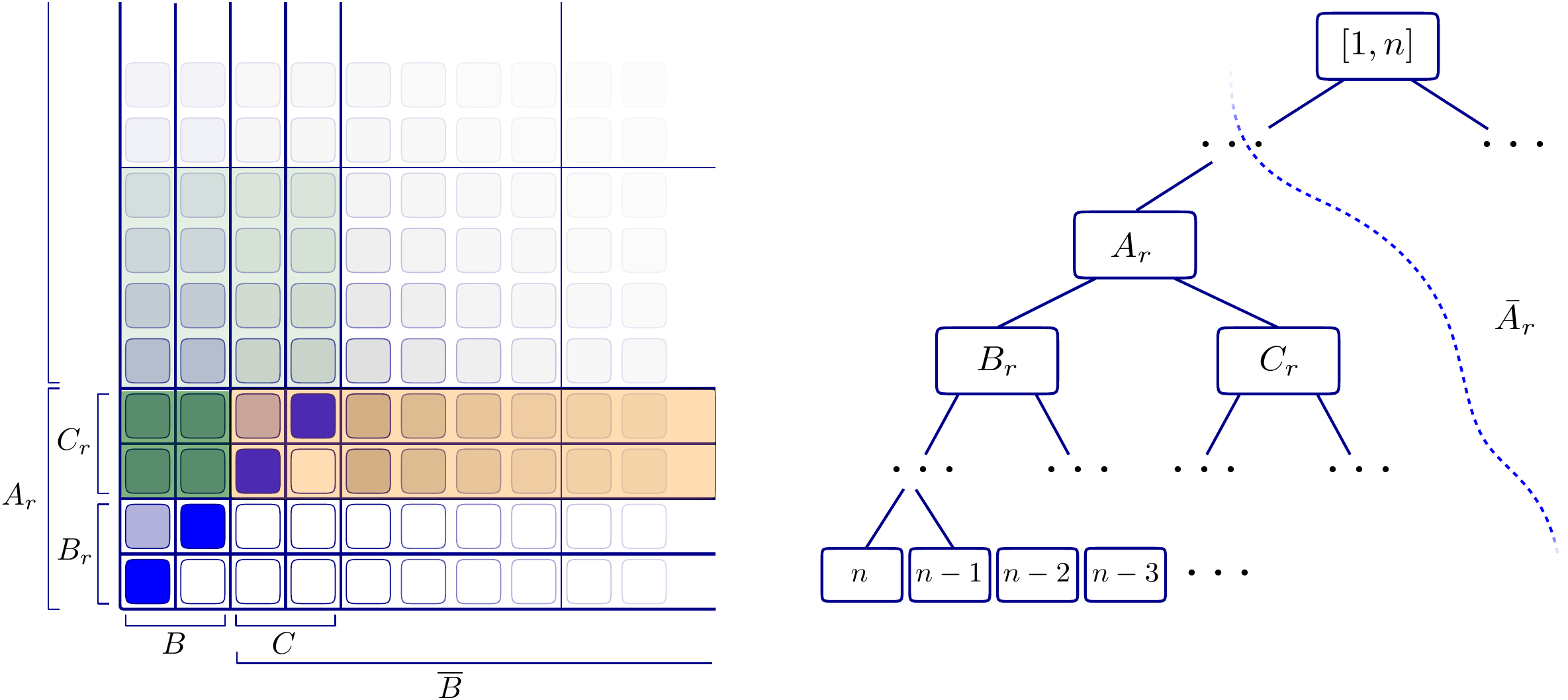}};
        \begin{scope}[x={(image.south east)}, y={(image.north west)}] 
            \node[anchor=center] (s1)  at (0.24,-0.02) {(a)};
            \node[anchor=center] (s2)  at (0.78,-0.02) {(b)};
\end{scope}
\end{tikzpicture}
\caption{\label{fig:row_algorithm}Batch updates in \cref{alg:row-tree}. (a)~When the recursive call on $B$ returns, the rows $B_r$ are reduced and their pivots sit in columns $B$. Before recursing on $C$ (rows $C_r$), we bring rows $C_r$ up to date in two steps: first we apply the internal matrix $\mat{A}_1$ to update the $B$ columns in rows $C_r$ (\cref{line:row-internal-update}, dark green), and then we multiply by $(\mat{P}^T\cdot\mat{\Lambda}\cdot\mat{P})[B, \overline{B}]$ to propagate those operations to all columns $\overline{B}$ to the right (\cref{line:row-external-update}, orange). Next, recursing on $C$ selects the pivot columns $C$ and reduces rows $C_r$. Finally, the call returns the combined transition matrix for $B \cup C$, which the parent call will use to update the block above $B_r \cup C_r$ (light green). (b)~The recursion tree on the rows $A_r = B_r \cup C_r$ (processed bottom-up).}
\end{center}
\end{figure}

To understand what must happen between the recursive call on $B$ and the recursive call on $C$, consider the state of the matrix when $B$ returns (\cref{fig:row_algorithm}):
\begin{itemize}
    \item The lower rows $B_r$ have been completely reduced, and their $k$ pivot columns have been permuted into column positions $B$.
    \item However, the delayed column operations generated while processing $B$ have \emph{not yet} been applied to the rows above $B_r$. In particular, we must apply them to the next block of rows $C_r$ before recursing on $C$, so that the leaves of $C$ see the correct entries when identifying $\lft(\cdot)$.
\end{itemize}
Why can we not simply multiply $(\mat{R} \cdot \mat{P})[C_r, B]$ by the coefficients $(\mat{P}^T \cdot \mat{\Lambda} \cdot \mat{P})[B, \overline{B}]$ in a single step?
Because during the recursion on $B$, the pivot columns in $B$ operated \emph{on one another}: earlier pivot columns in $B$ added multiples of themselves to later columns that subsequently became later pivot columns of $B$.
When each step $t \in B$ recorded its elimination multipliers in row $t$ of $\widehat{\mat{\Lambda}}$, those multipliers were computed relative to the \emph{already updated} pivot column at position $t$.
Meanwhile, in the untouched rows $C_r$ above $B_r$, the entries in columns $B$ still hold their values from before $B$ was called.

To determine the operations of $B$ on rows above $B_r$, each recursive call on $B$ therefore returns an upper-triangular $\matsize{k}{k}$ matrix $\mat{A}_1$ that summarizes the cumulative column operations performed \emph{among} the pivot columns $B$ themselves. We call these \emph{internal} updates. 
With $\mat{A}_1$, \cref{line:row-internal-update,line:row-external-update} of \cref{alg:row-tree} bring the rows $C_r$ up to date in two steps:
\begin{enumerate}
\item \textbf{Internal update on pivot columns $B$} (\cref{line:row-internal-update}, dark green in \cref{fig:row_algorithm}):
    \[
        (\mat{R} \cdot \mat{P})[C_r, B] \gets (\mat{R} \cdot \mat{P})[C_r, B] \cdot \mat{A}_1.
    \]
    This updates the entries of columns $B$ in rows $C_r$ so they reflect all operations among the pivots of $B$.
\item \textbf{Forward propagation to $\overline{B}$} (\cref{line:row-external-update}, orange in \cref{fig:row_algorithm}):
    \[
        (\mat{R} \cdot \mat{P})[C_r, \overline{B}] \gets (\mat{R} \cdot \mat{P})[C_r, \overline{B}] + (\mat{R} \cdot \mat{P})[C_r, B] \cdot (\mat{P}^T \cdot \mat{\Lambda} \cdot \mat{P})[B, \overline{B}],
    \]
    where $\overline{B} = [(i+j)/2, 2n]$ denotes all columns to the right of $B$. Using the newly updated block $(\mat{R} \cdot \mat{P})[C_r, B]$, this applies all column additions from the pivots of $B$ into every column to their right.
\end{enumerate}
Once these two updates are complete, the rows $C_r$ have applied all operations from every row below them.
We then recurse on $C$, which reduces the rows $C_r$, places their pivots into column positions $C$, and returns the $\matsize{k}{k}$ matrix $\mat{A}_2$ summarizing the internal operations among the pivot columns $C$.
Note that the operations performed during $C$ do not disturb the already reduced rows $B_r$ below it: the pivot columns chosen in $C$ are zero in rows $B_r$, and the column transpositions in $C$ only permute columns to the right of $B$.

Finally, before returning from $A$, we assemble the $\matsize{2k}{2k}$ transition matrix $\mat{T}_A$ that summarizes the combined internal operations on the pivot columns $A = B \cup C$ for any row above $A_r$ (such as the light green block in \cref{fig:row_algorithm}).
In the final permuted column order after $C$ returns, consider any untouched row above $A_r$ with entries $(x_B, x_C)$ in columns $(B, C)$:
\begin{itemize}
    \item Applying $B$'s internal update replaces $x_B$ with $x_B' = x_B \cdot \mat{A}_1$.
    \item Propagating $B$'s operations into the columns that end up in $C$ replaces $x_C$ with $x_C' = x_C + x_B' \cdot \widehat{\mat{\Lambda}}[B, C]$.
    \item Applying $C$'s internal update replaces $x_C'$ with $x_C'' = x_C' \cdot \mat{A}_2 = x_C \cdot \mat{A}_2 + x_B \cdot \mat{A}_1 \cdot \widehat{\mat{\Lambda}}[B, C] \cdot \mat{A}_2$.
\end{itemize}
Writing $(x_B', x_C'') = (x_B, x_C) \cdot \mat{T}_A$ yields the block upper-triangular return matrix in \cref{line:row-return}:
\[
    \mat{T}_A =
    \begin{pmatrix}
        \mat{A}_1 & \mat{A}_1 \cdot (\mat{P}^T \cdot \mat{\Lambda} \cdot \mat{P})[B, C] \cdot \mat{A}_2 \\
        \mat{0} & \mat{A}_2
    \end{pmatrix}.
\]

\begin{remark}
Here we pause to highlight the difference between returning from the left and right recursive calls respectively. In either case, a call returns the matrix required to bring the internal columns in the step up to date, i.e. if we applied the matrix to the columns of the original (permuted) boundary matrix, it would result in columns being  fully reduced. Now, when returning from a left child, we apply the newly reduced columns forward (so that the next row is reduced in the subsequent right calls). Alternatively, after a right child returns, the matrix is combined to create the update matrix for the entire span of columns (the union of the columns of the left and right calls).  
\end{remark}
\paragraph{Correctness and recovery of $\mat{U}$ and $\mat{V}$}
We now formalize the recursive invariant and show how to recover the matrices $\mat{U}$ and $\mat{V}$ of the lazy decomposition $\mat{R} = \bdr \cdot \mat{V}$.
When discussing a completed recursive call, all blocks are indexed using the permutation $\mat{P}$ at the moment that call returns.

\begin{lemma}\label{lem:fast-row-invariant}
Let $A=[i,j]$ be a recursive call, let
$A_r=\{n-t+1:t\in A\}$ be its rows, and let $J_A = \{j+1, \ldots, 2n\}$ be the columns to
the right of $A$ in the permuted matrix. Suppose that, on entry,
the rows $A_r$ have received all operations from previously
processed pivot positions $\{1, \ldots, i-1\}$. On return, the following hold:
\begin{enumerate}
    \item For every $t\in A$,
    \[
        \widehat{\mat{R}}[n-t+1,t]\ne 0,
        \qquad
        \widehat{\mat{R}}[n-t+1,\ell]=0
        \quad\text{for all }\ell>t.
    \]

    \item Let $\mat{T}_A$ be the matrix returned by the call.
    Consider a row above $A_r$, viewed in the permuted column
    order upon return, that has not yet received the column additions generated
    by this call. If its entries in $A$ and $J_A$ are $x_A$ and
    $x_{J_A}$, applying those additions gives
    \[
        x_A^{\mathrm{new}}=x_A\cdot\mat{T}_A,
        \qquad
        x_{J_A}^{\mathrm{new}}
          =x_{J_A}
           +x_A^{\mathrm{new}}\cdot
              \widehat{\mat{\Lambda}}[A,J_A].
    \]
    Columns to the left of $A$ are unchanged.
\end{enumerate}
\end{lemma}

\begin{proof}
We induct on $|A|$. If $A=\{i\}$ is a leaf, the entry assumption
implies that row $n-i+1$ has applied all operations from the rows below it,
so that every column at position $\ge i$ is zero in those lower rows.
The identity augmentation ensures that an eligible non-zero column
exists. The base case selects the earliest such column in the
original filtration order, places it at position $i$, records in
$\widehat{\mat{\Lambda}}[i, \ell]$ ($\ell > i$) the coefficients needed to clear the entries to the right, and sets
those entries to zero. This proves (1).
Because the pivot column at position $i$ is not modified in the base case (as the column is only applied to \emph{other columns to right}), the leaf
returns $\mat{T}_{\{i\}}=[1]$, and in any row above it, adding the recorded multiples of column $i$ to columns $\ell > i$ gives $x_i^{\mathrm{new}} = x_i \cdot 1$ and $x_{J_A}^{\mathrm{new}} = x_{J_A} + x_i^{\mathrm{new}} \cdot \widehat{\mat{\Lambda}}[i, J_A]$, proving (2).

Now suppose $|A|>1$, and split $A$ into its left and right children $B$ and $C$.
By the induction hypothesis on $B$, the recursive call on $B$ reduces rows $B_r$ and returns
$\mat{T}_B=\mat{A}_1$. Because its operations have not yet been applied to the
rows $C_r$, \cref{line:row-internal-update,line:row-external-update} apply the second assertion
 for $B$, where $J_B = \overline{B}$, to rows $C_r$: first updating
columns $B$ via $\mat{A}_1$, and then using those updated
entries and $\widehat{\mat{\Lambda}}[B,\overline{B}]$
to update all columns to the right of $B$. Thus, the entry assumption
holds for the recursive call on $C$.

By the induction hypothesis on $C$, the call on $C$ reduces rows $C_r$ and returns $\mat{T}_C=\mat{A}_2$. It does not disturb
the already reduced rows $B_r$, because the pivot columns chosen in $C$ lie to the right of $B$
and are zero in rows $B_r$, and the column permutations within $C$ only involve
positions to the right of $B$. This establishes (1) for $A = B \cup C$, and so the induction step.

For the induction step for (2), consider a row
above $A_r$ with entries $(x_B,x_C,x_{J_A})$ in the permuted order after $C$ returns.
Applying the operations from $B$ followed by those from $C$ yields
\[
\begin{aligned}
    x_B'   &= x_B\cdot\mat{A}_1,\\
    x_C'   &= x_C+x_B'\cdot\widehat{\mat{\Lambda}}[B,C],\\
    x_C''  &= x_C'\cdot\mat{A}_2,
\end{aligned}
\]
so the updated entries in $A = B \cup C$ are
\[
    (x_B',x_C'')=(x_B,x_C)
    \underbrace{\begin{pmatrix}
        \mat{A}_1 &
        \mat{A}_1\cdot
        \widehat{\mat{\Lambda}}[B,C]\cdot
        \mat{A}_2\\
        \mat{0} & \mat{A}_2
    \end{pmatrix}}_{\mat{T}_A},
\]
matching \cref{line:row-return}. Meanwhile, the entries in $J_A$ become
\[
    x_{J_A}^{\mathrm{new}}
       = x_{J_A}
         + x_B'\cdot\widehat{\mat{\Lambda}}[B,J_A]
         + x_C''\cdot\widehat{\mat{\Lambda}}[C,J_A]
       = x_{J_A}
         + (x_B',x_C'')\cdot
           \widehat{\mat{\Lambda}}[A,J_A],
\]
completing the induction.
\end{proof}

\begin{theorem}\label{thm:fast-row-correct}
\Cref{alg:row-tree} produces the lazy reduction of the augmented matrix.
Undo the column permutation via $\mat{R} = \widehat{\mat{R}} \cdot \mat{P}^T$ and restrict to the first $n$ columns to obtain the lazy reduction of $\bdr$.
\end{theorem}

\begin{proof}
Applying \cref{lem:fast-row-invariant} to the root call $A = [1, n]$ shows that all rows are reduced.
At each leaf, the entry assumption guarantees that the row has received all updates from lower rows, so the leaf sees the exact same entries as \cref{alg:row-incremental}, selects the same pivot column $\lft(n-i+1)$, and eliminates the entry in that row only from later eligible columns (\cref{rem:pivots}).
Thus \cref{alg:row-tree} performs the exact same column operations as \cref{alg:row-incremental}, grouped into block matrix products.
Finally, since the auxiliary identity columns occur after all original columns of $\bdr$, an auxiliary pivot is never used to reduce an original column, so restricting to the original columns gives the lazy reduction of $\bdr$.
\end{proof}

\begin{theorem}\label{thm:row-recover-u}
    The matrix $\mat{U}_{\mathrm{aug}} = 2\cdot \mat{I}_{2n} - \mat{\Lambda}$ satisfies $\mat{R}_{\mathrm{init}} = \mat{R}_{\mathrm{final}} \cdot \mat{U}_{\mathrm{aug}}$. Restricting $\mat{U}_{\mathrm{aug}}$ and $\mat{V}_{\mathrm{aug}} = \mat{U}_{\mathrm{aug}}^{-1}$ to the first $n$ rows and columns recovers the matrices $\mat{U}$ and $\mat{V}$ of the lazy decomposition $\mat{R} = \bdr \cdot \mat{V}$.
\end{theorem}
\begin{proof}
    In the lazy reduction (\cref{alg:standard,alg:row-incremental}), whenever a multiple $\alpha$ of pivot column $j$ is subtracted from a later column $j'$, column $j'$ has its current lowest non-zero in the pivot row of $j$ and has therefore not yet been used to reduce any other column.
    Consequently, at the moment of the update, row $j'$ of $\mat{U}$ has a single non-zero entry $\mat{U}[j', j'] = 1$ on the diagonal, and the row update $\mat{U}[j, \cdot] \gets \mat{U}[j, \cdot] + \alpha \cdot \mat{U}[j', \cdot]$ simply sets the off-diagonal entry $\mat{U}[j, j'] = \alpha$.

    By \cref{thm:fast-row-correct}, \cref{alg:row-tree} performs the same lazy reduction, recording $-\alpha$ in the corresponding off-diagonal entry of $\mat{\Lambda}$ while keeping ones on the diagonal of $\mat{\Lambda}$.
    Thus $\mat{U}_{\mathrm{aug}}$ agrees with $2\cdot \mat{I}_{2n} - \mat{\Lambda}$ both on the diagonal (where both equal $1$) and off the diagonal (where $\mat{U}_{\mathrm{aug}}$ has $+\alpha$ and $\mat{\Lambda}$ has $-\alpha$).
    Because the auxiliary columns sit to the right of the original columns and never reduce them, $\mat{U}_{\mathrm{aug}}$ is upper-triangular with respect to the original filtration order on the first $n$ columns, and its leading $\matsize{n}{n}$ principal block is $\mat{U}$ with inverse $\mat{V} = \mat{U}^{-1}$.
\end{proof}

\begin{remark}
Since $\mat{U}$ is upper-triangular, $\mat{V} = \mat{U}^{-1}$ can be computed in $O(n^{\matexp})$ time by \cref{thm:matrix_inv}.
In fact, $\mat{V}$ can be recovered directly from the output of \cref{alg:row-tree} \emph{without} a separate matrix inversion.
Let $M = 2n$ be the number of columns of the augmented matrix, let $\mat{P}$ be the final permutation, and write
\[
    \widehat{\mat{N}}
        =\mat{P}^T\cdot(\mat{\Lambda}-\mat{I}_M)\cdot\mat{P},
    \qquad A=\{1,\ldots,n\},\qquad Z=\{n+1,\ldots,M\}.
\]
Because the columns in $Z$ are never used as pivot columns, the bottom rows of $\widehat{\mat{N}}$ vanish:
\[
    \widehat{\mat{N}}=
    \begin{pmatrix}
        \mat{N}_{AA} & \mat{N}_{AZ}\\
        \mat{0} & \mat{0}
    \end{pmatrix},
\]
where $\mat{N}_{AA}=\widehat{\mat{N}}[A,A]$ is strictly upper-triangular and
$\mat{N}_{AZ}=\widehat{\mat{N}}[A,Z]$.
Notice that the return formula in \cref{line:row-return} of \cref{alg:row-tree} is precisely the standard recursive block-inversion formula for the unit upper-triangular matrix $\mat{I}_n - \mat{N}_{AA}$: each leaf returns $1$, and splitting an interval into $B$ and $C$ forms the off-diagonal block $\mat{T}_B \cdot \widehat{\mat{N}}[B,C] \cdot \mat{T}_C$.
Thus the matrix returned by the root call $A = [1, n]$ is already $\mat{T}_A = (\mat{I}_n - \mat{N}_{AA})^{-1}$.
In the final permuted order, \cref{thm:row-recover-u} gives
\[
    \widehat{\mat{U}}=\mat{I}_M-\widehat{\mat{N}}
    =
    \begin{pmatrix}
        \mat{I}_n-\mat{N}_{AA} & -\mat{N}_{AZ}\\
        \mat{0} & \mat{I}_{M-n}
    \end{pmatrix},
    \qquad
    \widehat{\mat{V}}= \widehat{\mat{U}}^{-1} =
    \begin{pmatrix}
        \mat{T}_A & \mat{T}_A\cdot\mat{N}_{AZ}\\
        \mat{0} & \mat{I}_{M-n}
    \end{pmatrix}.
\]
Undoing the permutation via $\mat{V}_{\mathrm{aug}}=\mat{P}\cdot\widehat{\mat{V}}\cdot\mat{P}^T$ and taking the leading $\matsize{n}{n}$ block indexed by the original columns recovers $\mat{V}$ using a single additional matrix multiplication $\mat{T}_A \cdot \mat{N}_{AZ}$ in $O(n^{\matexp})$ time.
\end{remark}

\paragraph{Running Time Analysis} 
At a leaf ($|A| = 1$), scanning row $n - i + 1$ to find $\lft(n - i + 1)$, updating $\mat{P}$, recording the multipliers in row $i$ of $\widehat{\mat{\Lambda}}$, and zeroing out the row each take $O(n)$ time.

At an internal node where $|B| = |C| = k$, the work outside the two recursive calls consists of:
\begin{enumerate}
\item Internal column update (\cref{line:row-internal-update}): Multiplying the $\matsize{k}{k}$ block $(\mat{R} \cdot \mat{P})[C_r, B]$ by the $\matsize{k}{k}$ matrix $\mat{A}_1$ takes $O(k^\matexp)$ time.
\item Forward row update (\cref{line:row-external-update}): Multiplying the $\matsize{k}{k}$ block $(\mat{R} \cdot \mat{P})[C_r, B]$ by the $\matsize{k}{|\overline{B}|}$ matrix $(\mat{P}^T \cdot \mat{\Lambda} \cdot \mat{P})[B, \overline{B}]$ (where $|\overline{B}| \le 2n$) takes $O(n k^{\matexp - 1})$ time by \cref{lem:block_multiplication}, and adding the result to $(\mat{R} \cdot \mat{P})[C_r, \overline{B}]$ takes $O(nk)$ time. Applying and maintaining the column permutations on these blocks also takes $O(nk)$ time by \cref{lem:perm}.
\item Assembling the return matrix (\cref{line:row-return}): Extracting the $\matsize{k}{k}$ submatrix $(\mat{P}^T \cdot \mat{\Lambda} \cdot \mat{P})[B, C]$ takes $O(k^2)$ time, and computing the product $\mat{A}_1 \cdot (\mat{P}^T \cdot \mat{\Lambda} \cdot \mat{P})[B, C] \cdot \mat{A}_2$ of three $\matsize{k}{k}$ matrices takes $O(k^\matexp)$ time.
\end{enumerate}
Since $k \le n$, the non-recursive work at a node of size $2k$ is bounded by $O(n k^{\matexp - 1})$.
This yields the recurrence $T(2k) = 2 T(k) + O(n k^{\matexp - 1})$ with base case $T(1) = O(n)$, which is identical to the recurrence for the column algorithm and solves to $O(n^{\matexp})$ time for $\matexp > 2$ (and $O(n^2 \log n)$ time if $\matexp = 2$).
\begin{remark}
    As a final comment on the fast row algorithm, we note that the augmentation is not required (as there is no matrix inversion), but it does complicate the base case so we opt to include it. 
\end{remark}
\section{Discussion}
In this paper, we have shown that in the  case of persistent (co)homology, the standard bases returned by the exhaustive and lazy reductions can be computed in matrix multiplication time. It is worth noting that while both the columns and row algorithms utilize batch updates, they are qualitatively different:
%
 %
(1) Only the column algorithm critically relies  on fast  matrix inversion; 
(2) The row algorithm update sequence is substantially more complex 
but  requires only column permutations. 
%
While one can use the column algorithm to compute the lazy basis and the row algorithm for the exhaustive basis in the same asymptotic running time, the resulting algorithms are both complex and not particularly illuminating. There are also numerous technical complications to deal with, so we omit them from this paper. Another outstanding issue is whether it is possible to transform the exhaustive representatives to lazy representatives and vice versa without redoing the reduction. Again, we believe this possible, but omit the details due to space and because the resulting asymptotic running time is the same. Finally, we conclude with an open question: what is the optimal running time for computing representatives in  zigzag persistence. Recently, this has  been partially addressed in \cite{dey2021updating,dey2024fast,dey2025apex} but optimality in full generality remains open.


\section*{Acknowledgement}
Dmitriy Morozov was supported by the U.S.\ Department of Energy, Office of Science, Office of Advanced Scientific Computing Research's Applied Mathematics Competitive Portfolios program under Contract No.\ DE-AC02-05CH11231.
Primoz Skraba was supported by the EPSRC AI Hub on Mathematical Foundations of Intelligence: An ``Erlangen Programme" for AI No. EP/Y028872/1.

\bibliographystyle{unsrturl}
\bibliography{bibliography}

@inproceedings{milosavljevic2011zigzag,
  title={Zigzag persistent homology in matrix multiplication time},
  author={Milosavljevi{\'c}, Nikola and Morozov, Dmitriy and Skraba, Primoz},
  booktitle={Proceedings of the twenty-seventh Annual Symposium on Computational Geometry},
  pages={216--225},
  year={2011}
}

@article{strassen1969gaussian,
  title={Gaussian elimination is not optimal},
  author={Strassen, Volker},
  journal={Numerische mathematik},
  volume={13},
  number={4},
  pages={354--356},
  year={1969},
  publisher={Springer}
}

@article{bunch1974triangular,
  title={Triangular factorization and inversion by fast matrix multiplication},
  author={Bunch, James R and Hopcroft, John E},
  journal={Mathematics of Computation},
  volume={28},
  number={125},
  pages={231--236},
  year={1974}
}

@inproceedings{zomorodian2004computing,
  title={Computing persistent homology},
  author={Zomorodian, Afra and Carlsson, Gunnar},
  booktitle={Proceedings of the twentieth annual symposium on Computational geometry},
  pages={347--356},
  year={2004}
}

@article{edelsbrunner2002topological,
  title={Topological persistence and simplification},
  author={Edelsbrunner, H. and Letscher, D. and Zomorodian, A.},
  journal={Discrete \& computational geometry},
  volume={28},
  pages={511--533},
  year={2002},
  publisher={Springer}
}

@inproceedings{cohen2006vines,
  title={Vines and vineyards by updating persistence in linear time},
  author={Cohen-Steiner, David and Edelsbrunner, Herbert and Morozov, Dmitriy},
  booktitle={Proceedings of the twenty-second annual symposium on Computational geometry},
  pages={119--126},
  year={2006}
}

@article{bauer2021ripser,
  title={Ripser: efficient computation of Vietoris--Rips persistence barcodes},
  author={Bauer, Ulrich},
  journal={Journal of Applied and Computational Topology},
  volume={5},
  number={3},
  pages={391--423},
  year={2021},
  publisher={Springer}
}

@article{bauer2017phat,
  title={Phat--persistent homology algorithms toolbox},
  author={Bauer, Ulrich and Kerber, Michael and Reininghaus, Jan and Wagner, Hubert},
  journal={Journal of symbolic computation},
  volume={78},
  pages={76--90},
  year={2017},
  publisher={Elsevier}
}

@inproceedings{bauer2014clear,
  title={Clear and compress: Computing persistent homology in chunks},
  author={Bauer, Ulrich and Kerber, Michael and Reininghaus, Jan},
  booktitle={Topological Methods in Data Analysis and Visualization III: Theory, Algorithms, and Applications},
  pages={103--117},
  year={2014},
  organization={Springer}
}

@inproceedings{kerber2016persistent,
  title={Persistent homology and nested dissection},
  author={Kerber, Michael and Sheehy, Donald R and Skraba, Primoz},
  booktitle={Proceedings of the Twenty-Seventh Annual ACM-SIAM Symposium on Discrete Algorithms},
  pages={1234--1245},
  year={2016},
  organization={SIAM}
}

@InProceedings{morozov24,
  author =	{Morozov, Dmitriy and Skraba, Primoz},
  title =	{{Persistent (Co)Homology in Matrix Multiplication Time}},
  booktitle =	{41st International Symposium on Computational Geometry (SoCG 2025)},
  pages =	{68:1--68:16},
  series =	{LIPIcs},
  ISBN =	{978-3-95977-370-6},
  ISSN =	{1868-8969},
  year =	{2025},
  volume =	{332},
  editor =	{Aichholzer, Oswin and Wang, Haitao},
  publisher =	{Schloss Dagstuhl -- Leibniz-Zentrum f{\"u}r Informatik},
  address =	{Dagstuhl, Germany},
  doi =		{10.4230/LIPIcs.SoCG.2025.68}
}

@article{dey2024fast,
  author  = {Dey, Tamal K. and Hou, Tao and Morozov, Dmitriy},
  title   = {A Fast Algorithm for Computing Zigzag Representatives},
  journal = {Algorithmica},
  volume  = {88},
  year    = {2026},
  note    = {Article 36},
  doi     = {10.1007/s00453-026-01386-4},
  url     = {https://doi.org/10.1007/s00453-026-01386-4}
}

@inproceedings{dey2021updating,
  title     = {Updating Zigzag Representatives Efficiently},
  author    = {Dey, Tamal K. and Hou, Tao and Morozov, Dmitriy},
  booktitle = {Proceedings of the Annual ACM-SIAM Symposium on Discrete Algorithms (SODA)},
  year      = {2027},
  url       = {https://arxiv.org/abs/2607.16153}
}

@inproceedings{dey2022fast,
  title={Fast Computation of Zigzag Persistence},
  author={Dey, Tamal K and Hou, Tao},
  booktitle={30th Annual European Symposium on Algorithms (ESA 2022)},
  year={2022},
  organization={Schloss Dagstuhl-Leibniz-Zentrum f{\"u}r Informatik}
}

@inproceedings{dey2025apex,
  author    = {Dey, Tamal K. and Hou, Tao and Morozov, Dmitriy},
  title     = {{Apex Representatives}},
  booktitle = {International Symposium on Computational Geometry (SoCG)},
  series    = {Leibniz International Proceedings in Informatics (LIPIcs)},
  volume    = {332},
  pages     = {40:1--40:16},
  year      = {2025},
  doi       = {10.4230/LIPIcs.SoCG.2025.40},
  url       = {https://doi.org/10.4230/LIPIcs.SoCG.2025.40}
}

@inproceedings{busaryev2012annotating,
  title={Annotating simplices with a homology basis and its applications},
  author={Busaryev, Oleksiy and Cabello, Sergio and Chen, Chao and Dey, Tamal K and Wang, Yusu},
  booktitle={Scandinavian workshop on algorithm theory},
  pages={189--200},
  year={2012},
  organization={Springer}
}

@inproceedings{hylton2017performance,
  title={Performance enhancement of a computational persistent homology package},
  author={Hylton, Alan and Henselman-Petrusek, Greg and Sang, Janche and Short, Robert},
  booktitle={2017 IEEE 36th international performance computing and communications conference (IPCCC)},
  pages={1--8},
  year={2017},
  organization={IEEE}
}

@book{edelsbrunner2022computational,
  title={Computational topology: an introduction},
  author={Edelsbrunner, Herbert and Harer, John L},
  year={2022},
  publisher={American Mathematical Society}
}

@article{morozov2005persistence,
  title={Persistence algorithm takes cubic time in worst case},
  author={Morozov, Dmitriy},
  journal={BioGeometry News, Dept. Comput. Sci., Duke Univ},
  volume={2},
  year={2005},
  publisher={Citeseer}
}

@article{de2011dualities,
  title={Dualities in persistent (co) homology},
  author={De Silva, Vin and Morozov, Dmitriy and Vejdemo-Johansson, Mikael},
  journal={Inverse Problems},
  volume={27},
  number={12},
  pages={124003},
  year={2011},
  publisher={IOP Publishing}
}

@ARTICLE{CSLV22,
  title     = "Lexicographic optimal homologous chains and applications to point
               cloud triangulations",
  author    = "Cohen-Steiner, David and Lieutier, André and Vuillamy, Julien",
  journal   = "Discrete \& computational geometry",
  publisher = "Springer Science and Business Media LLC",
  volume    =  68,
  number    =  4,
  pages     = "1155--1174",
  month     =  dec,
  year      =  2022,
  //url       = "https://idp.springer.com/authorize/casa?redirect_uri=https://link.springer.com/article/10.1007/s00454-022-00432-6&casa_token=Q-grNMDQHbMAAAAA:eD9llwR8cBRwBpYXASiGB5pWaE4lgGshmLmFTK1yKA7WI0_tXD9AP5P2y6oZJvnSYzsA8-sYZYIPPsQd",
  doi       = "10.1007/s00454-022-00432-6",
  issn      = "0179-5376,1432-0444",
  language  = "en"
}

@ARTICLE{NiMo24,
  title     = "Topological optimization with big steps",
  author    = "Nigmetov, Arnur and Morozov, Dmitriy",
  journal   = "Discrete \& computational geometry",
  publisher = "Springer Science and Business Media LLC",
  volume    =  72,
  number    =  1,
  pages     = "310--344",
  month     =  jul,
  year      =  2024,
  //url       = "https://link.springer.com/article/10.1007/s00454-023-00613-x",
  doi       = "10.1007/s00454-023-00613-x",
  issn      = "0179-5376,1432-0444",
  language  = "en"
}

@ARTICLE{dSMVJ11,
  title   = "Persistent Cohomology and Circular Coordinates",
  author  = "de Silva, Vin and Morozov, Dmitriy and Vejdemo-Johansson, Mikael",
  journal = "Discrete \& computational geometry",
  volume  =  45,
  number  =  4,
  pages   = "737--759",
  month   =  jun,
  year    =  2011,
  url     = "https://doi.org/10.1007/s00454-011-9344-x",
  doi     = "10.1007/s00454-011-9344-x",
  issn    = "0179-5376,1432-0444"
}

@ARTICLE{henselmanghristl6,
author= {{Henselman}, G. and {Ghrist}, R.},
title= "{Matroid Filtrations and Computational Persistent Homology}",
journal= {ArXiv e-prints},
archivePrefix = "arXiv",
eprint = {1606.00199},
primaryClass = "math.AT",
year= 2016,
month= jun,
adsurl = {http://adsabs.harvard.edu/abs/2016arXiv160600199H}
}

@inproceedings{chen2011output,
  title={An output-sensitive algorithm for persistent homology},
  author={Chen, Chao and Kerber, Michael},
  booktitle={Proceedings of the twenty-seventh annual symposium on Computational geometry},
  pages={207--216},
  year={2011}
}

@incollection{wagner2011efficient,
  title={Efficient computation of persistent homology for cubical data},
  author={Wagner, Hubert and Chen, Chao and Vu{\c{c}}ini, Erald},
  booktitle={Topological methods in data analysis and visualization II: theory, algorithms, and applications},
  pages={91--106},
  year={2011},
  publisher={Springer}
}

@article{kalai1983enumeration,
  title={Enumeration of Q-acyclic simplicial complexes},
  author={Kalai, Gil},
  journal={Israel Journal of Mathematics},
  volume={45},
  pages={337--351},
  year={1983},
  publisher={Springer}
}

@article{skraba2020randomly,
  title={Randomly Weighted $ d $-Complexes: Minimal Spanning Acycles and Persistence Diagrams},
  author={Skraba, Primoz and Thoppe, Gugan and Yogeshwaran, D},
  journal={The Electronic Journal of Combinatorics},
  pages={P2--11},
  year={2020}
}

@INPROCEEDINGS{CdSM09,
  title     = "Zigzag Persistent Homology and Real-valued Functions",
  author    = "Carlsson, Gunnar and de Silva, Vin and Morozov, Dmitriy",
  booktitle = "Proceedings of the Twenty-fifth Annual Symposium on Computational
               Geometry",
  publisher = "ACM",
  address   = "New York, NY, USA",
  pages     = "247--256",
  series    = "SCG '09",
  year      =  2009,
  //url       = "http://doi.acm.org/10.1145/1542362.1542408",
  doi       = "10.1145/1542362.1542408",
  isbn      =  9781605585017
}

@ARTICLE{CdS10,
  title   = "Zigzag Persistence",
  author  = "Carlsson, Gunnar and de Silva, Vin",
  journal = "Foundations of computational mathematics",
  volume  =  10,
  number  =  4,
  pages   = "367--405",
  month   =  aug,
  year    =  2010,
  url     = "https://doi.org/10.1007/s10208-010-9066-0",
  doi     = "10.1007/s10208-010-9066-0",
  issn    = "1615-3375,1615-3383"
}

\appendix

\section{Matrix Inversion in Matrix Multiplication Time}\label{appendix:fastinversion}

We present a simplified proof for the special case of triangular matrices. This can be shown with greater generality and was stated as Fact 4 in Strassen's original
paper on sub-cubic matrix multiplication time \cite{strassen1969gaussian}. While Strassen's result required certain submatrices to be non-singular, these constraints where removed by Bunch and Hopcroft \cite{bunch1974triangular}. We include the proof for completeness.  Importantly, this makes no assumption on the underlying field. 
We provide the proof for lower triangular matrices, but it holds equivalently for upper triangular, by transposition. 
 We begin with the following observation.
 \begin{observation}
 A  lower triangular matrix is non-singular if and only if the diagonal elements are non-zero.
 \end{observation}
 This follows from the fact that the columns must be linearly independent and so the matrix is full rank and hence invertible. 

 \thmmatrixinv*
 \begin{proof}
  Consider a non-singular lower triangular  $\matsize{n}{n}$ matrix $\mat{A}$.  Further assume that $n$ is a power of 2. If it is not we can consider the matrix 
 \[\begin{bmatrix}\mat{A} & \mat{0}_{\matsize{n}{m}}\\ \mat{0}_{\matsize{m}{n}} & \mat{I}_{m}\end{bmatrix}\]
 where $m=2^{\lceil \log_2 n\rceil} - n$. 
From now on assuming $n$ is a power of 2, we can write $\mat{A}$ as
 \[
     \begin{bmatrix}
         \mat{B}&\mat{0} \\
         \mat{C} &\mat{D}
     \end{bmatrix}\]
 where all sub-matrices are of size $\frac{n}{2}$.
 Observe that $\mat{B}$ and $\mat{D}$ lower diagonal with a non-zero diagonal elements  -- the same as $\mat{A}$ -- and so are non-singular. Through direct computation, we note that the inverse is 
 \[\mat{A}^{-1} = 
     \begin{bmatrix}
         \mat{B}^{-1}&\mat{0} \\
         -\mat{D}^{-1}\mat{C}\mat{B}^{-1} &\mat{D}^{-1}  
     \end{bmatrix}\]
 That this is the inverse can be directly verified by expanding $\mat{A}^{-1} \cdot \mat{A}$ or vice versa. Therefore, we can compute the inverse by computing the inverse of $\mat{B}$ and $\mat{D}$.
 Therefore, we require the inversion of two $\matsize{\frac{n}{2}}{\frac{n}{2}}$ matrices, and two matrix multiplications. 
 As $\mat{B}$ and $\mat{D}$ are both lower triangular with a non-zero diagonal, we can further subdivide and compute the inverse as above but with the initial matrix being of size $\matsize{\frac{n}{2}}{\frac{n}{2}}$.

To bound the total computational complexity, we let $I(n)$ be the time to invert an $\matsize{n}{n}$ (lower triangular) matrix and $M(n)$ the time to perform matrix multiplication of two $\matsize{n}{n}$ matrices. This gives the recurrence
 \[I(n) \leq 2 I\left(\frac{n}{2}\right) + 2 M\left(\frac{n}{2}\right).\]
Let  $M(n) \leq c_1 n^{\matexp}$ with $M(1)=c_1$ and $I(1)=c_2$. Assuming $n = 2^{k}$, we can write 
\begin{align*}
 I(2^k) &\leq  2I(2^{k-1}) +2M(2^{k-1})\\
 &\leq  2^2 (I(2^{k-2})+M(2^{k-2}) +2M(2^{k-1})\\
 &\qquad\qquad\qquad\vdots\\
 &\leq  2^k I(1) + \sum\limits_{j=1}^{k}2^jM(2^{k-j})\\
 &\leq  c_2 2^k  + c_1 \sum\limits_{j=1}^{k-1}2^j 2^{(k-j)\matexp}\\
 &=  c_2 2^k  + c_1 2^{k\matexp} \sum\limits_{j=1}^{k-1}2^{-j(\matexp-1)}
 \end{align*}
 We observe that since $\matexp \geq 2$,
 \begin{align*}
 \sum\limits_{j=1}^{k-1}2^{-j(\matexp-1)} \leq  \sum\limits_{j=0}^{\infty}2^{-j(\matexp-1)}  \leq 2
 \end{align*}
 Substituting in and  using that $n=2^k$, we get
 \begin{align*}
 I(2^k) &\leq  c_2 2^k  + c_1 2^{k\matexp} \sum\limits_{j=1}^{k-1}2^{-j(\matexp-1)}\\
 &\leq c_2 2^k  + 2 c_1 (2^k)^\matexp = O(n+ n^\matexp) = O(n^\matexp).
 \end{align*}
 
\end{proof}

\section{Analysis of Persistence Recursion}
\label{appendix:recursion-analysis}

For all the algorithms, the recursion is of the form
\[ \PH(\ell) \leq  2\PH\left(\frac{\ell}{2}\right) +c_1 n \left(\frac{\ell}{2}\right)^{\matexp-1} \]
where $\PH(n)$ is the running time for persistent (co)homology on a filtration with $n$ cells and
 $c_1$ takes into account the constant in terms of number of applications as well as for the matrix multiplication. 
Assuming $n$ is a power of 2, for $\ell = 2^k$,
\begin{align*}
 \PH(2^k) &\leq  2\PH(2^{k-1}) +c_1 n 2^{(k-1)(\matexp-1)} \\
 &\leq  2^2 \PH(2^{k-2})+2 c_1 n 2^{(k-2)(\matexp-1)} +c_1 n 2^{(k-1)(\matexp-1)}\\
 &\qquad\qquad\qquad\vdots\\
 &\leq  2^k \PH(1) + c_1 n \sum\limits_{j=0}^{k-1}2^{j}2^{(k-1-j)(\matexp-1)}\\
 & =  2^k \PH(1) + c_1 n 2^{(k-1)(\matexp-1)} \sum\limits_{j=0}^{k-1}2^{j-j(\matexp-1)}\\
 & =  2^k \PH(1) + c_1 n 2^{(k-1)(\matexp-1)} \sum\limits_{j=0}^{k-1}2^{-j(\matexp-2)}
 \end{align*}
Assuming $\matexp>2$, we note that 
\begin{align*}
\sum\limits_{j=0}^{k-1}2^{-j(\matexp-2)} \leq \sum\limits_{j=0}^{\infty}2^{-j(\matexp-2)} 
= \frac{1}{1-2^{-(\matexp-2)}} \leq c_2
 \end{align*}
 Observe that $c_2$ depends on $\matexp$.
 Substituting in, we get
 \begin{align*}
 \PH(2^k) &\leq  2^k \PH(1) + c_1 n 2^{(k-1)(\matexp-1)} \sum\limits_{j=0}^{k-1}2^{-j(\matexp-2)} \\
 &\leq 2^k \PH(1) + c_1 c_2 n 2^{(k-1)(\matexp-1)}
 \end{align*}
If $\ell = n = 2^k$, then
\begin{align*}
\PH(n) \leq n \PH(1) +\frac{c_1 c_2}{2} n^{\matexp} 
\end{align*}
Since $\PH(1)\leq  c_3 n$, we get
\[\PH(n) \leq O(n^2 + n^{\matexp}) = O(n^{\matexp})\]
as required. 

If $\matexp=2$, then 
\begin{align*}
\sum\limits_{j=0}^{k-1}2^{-j(\matexp-2)} =  \sum\limits_{j=0}^{k-1} 1 
 =   k  \leq c_4 \log n
 \end{align*}
 Substituting in, we get 
 \begin{align*}
 \PH(2^k) &\leq  2^k \PH(1) + c_1 n 2^{(k-1)} \sum\limits_{j=0}^{k-1}2^{-j(\matexp-2)} \\
 &\leq 2^k \PH(1) + c_1 c_2 c_4 n 2^{(k-1)} \log n\\
 &=  2^k \PH(1) + \frac{c_1 c_2 c_4}{2} n 2^{k} \log n\\
 & \leq O(n^{2}+n^2 \log n ) =  O(n^2 \log n) 
 \end{align*}

\end{document}